\documentclass{amsart}
\usepackage[T1]{fontenc}
\usepackage[utf8x]{inputenc}

\usepackage{amscd,amsmath,amssymb,amsfonts,epsfig,graphics,amsthm}
\usepackage{booktabs}
\usepackage{graphicx}
\usepackage{epstopdf}
\usepackage{mathabx} 
\usepackage{nccmath} 

\usepackage{hyperref}
\usepackage[authoryear]{natbib}
\usepackage{xcolor}

\let\oldcite\cite
\renewcommand{\cite}[1]{\textcolor{blue}{\oldcite{#1}}}

\let\oldcitep\citep
\renewcommand{\citep}[1]{\textcolor{blue}{\oldcitep{#1}}}

\let\oldcitet\citet
\renewcommand{\citet}[1]{\textcolor{blue}{\oldcitet{#1}}}
\usepackage{mdframed} 
\usepackage{lipsum}
\usepackage{tcolorbox} 
\usepackage{comment}  

\tcbuselibrary{theorems}
\usepackage{MnSymbol}

\usepackage[margin=3cm]{geometry}
\newtheorem{theorem}{Theorem}[section]
\vspace{10pt}
\newtheorem{proposition}[theorem]{Proposition}
\newtheorem{prop}[theorem]{Proposition}
\newtheorem{lemma}[theorem]{Lemma}

\theoremstyle{definition}
\newtheorem{definition}[theorem]{Definition}
\vspace{10pt}

\newtheorem{example}[theorem]{Example}
\theoremstyle{remark}
\newtheorem{remark}[theorem]{Remark}

\newtheorem{defremark}{Definition/Remark}

\newtcbtheorem{mytheo}{Theorem}%
{colback=green!5,colframe=black!35!black,fonttitle=\bfseries}{th}

\newcommand{\nd}{\noindent}
\newcommand{\norm}[1]{\left\lVert#1\right\rVert}

\newcommand{\cl}{\mathbb{R}}

\newcommand{\ddim}{\mathbb{R}^d}
\newcommand{\Ddim}{\mathbb{R}^D}

\newcommand{\bbn}{\mathbb{N}}

\newcommand{\df}{\coloneqq}

\newcommand{\coloneqq}{:=}
\newcommand{\setR}{\mathbb{R}}

\makeatletter

\makeatletter

\def\cC{\mathcal C}

\begin{document}
\pagenumbering{arabic}

\title[Graph Laplacian asymptotics near isolated singularities]{\large \textbf{Asymptotics of the graph Laplacian near an isolated singularity}}
\date{\today}

\maketitle
\begin{center}
  \textbf{Susovan Pal} \\
  Department of Mathematics and Data Science, Vrije Universiteit Brussel (VUB) \\
  Pleinlaan 2, B-1050 Elsene/Ixelles, Belgium \\
  \texttt{susovan.pal@vub.be, susovan97@gmail.com}
  
\end{center}
\begin{abstract}
    \nd We investigate the asymptotics of the graph Laplacian on a Riemannian manifold that admits an isolated singularity $x$. This operator, and its variants underlie kernel methods, manifold learning, and spectral constructions based on sampled data. We show that under suitably controlled curvature growth near $x$, the graph Laplacian at $x$ converges to the weighted Laplace--Beltrami operator as the bandwidth $t$ approaches zero. On the other hand, under a conformal change of the Riemannian metric whose conformal factor depends only on the angular variable near $x$, we show that the curvature growth condition is not satisfied, and the graph Laplacian grows like $O\!\left(\frac{1}{\sqrt{t}}\right)$ as $t\downarrow 0$. Numerical simulations validate these asymptotics.
\end{abstract}
\\
\smallskip
\nd \textbf{Word count:} 10,350.

\section{Introduction}

\subsection{Graph Laplacian and its Continuous Limit}

Let $(M,d,\mu)$ be a metric measure space where $d$ is a metric on $M$ and $\mu$ is a \textit{non-atomic} Borel measure on $M$. Assume that the metric measure space has a finite Hausdorff dimension, denoted again by $d$.
In particular, $(M,d,\mu)$ may arise as a smooth $d$-dimensional Riemannian manifold $(M,g)$ endowed with its geodesic distance $d_g$ and volume measure $\mathrm{vol}_g$ induced by the Riemannian metric $g$, such that $\mu = \mathrm{vol}_g$.  
It is assumed that $\mu(M)< \infty.$ 

Let $(\Omega, \mathcal{F}, \mathbb{P})$ be a probability space with probability measure $\mathbb{P}$, and let $X\colon (\Omega, \mathcal{F}, \mathbb{P})\to (M,d,\mu)$ be a random variable, that is, a Borel measurable function. Let $P_X$ be the pushforward of $\mathbb{P}$ to $M$ by $X$, that is, $P_X(E):=\mathbb{P}(X^{-1}(E))$ for all $E \in \mathcal{B}(M)$, the Borel sets of $M$.
Assume that $P_X$ has a continuous density $p$ with respect to $\mu$, that we will assume to be smooth in case the metric measure space becomes a smooth manifold.
Let $(X_i)_{i=1}^n$ be i.i.d.\ $M$-valued sampled from $P_X$ and let $k_t$ denote the Gaussian kernel
\[
k_t(x,y) = \exp\!\left(-\frac{d(x,y)^2}{t}\right),
\]
where $t>0$ denotes the bandwidth parameter. All dimensional constants are computed under this convention.

\paragraph{\textbf{Unnormalized graph Laplacian: discrete and continuous.}}
Following the graph Laplacian framework used in \citet{PalTewodrose2025} and \citet{BelkinQueWangZhou2012} for Riemannian manifolds, possibly with boundary or singularities, we define the \emph{unnormalized graph Laplacian} on the metric measure space acting on a bounded measurable function $f\colon M\to\mathbb{R}$ by
\begin{equation}\label{eq:discrete_GL}
L_{n,t}f(x)
= \frac{1}{n\,t^{d/2+1}}\sum_{j=1}^n 
  e^{-\frac{d(x,X_j)^2}{t}}\big(f(x)-f(X_j)\big),
\end{equation}
where $d(.,.)$ is the given metric $d(.,.)$ on $M$, and $d\in \mathbb{N}$ is the Hausdorff dimension of the underlying metric space $(M,d).$ The corresponding deterministic—or continuous- graph Laplacian is obtained by taking the expectation with respect to $P_X$. Denoting the density (i.e. Radon-Nikodym derivative) of $P_X$ with respect to $\mu$ by $p$, we write:
\begin{definition}[\textbf{continuous graph Laplacian}]\label{eq:continuous_GL}
\[
L_t f(x)
= \frac{1}{t^{d/2+1}} \int_M
  e^{-\frac{d(x,y)^2}{t}}\big(f(x)-f(y)\big) p(y)\,d\mu(y).
\]
\end{definition}
By the strong law of large numbers, the random operator $L_{n,t}f(x)$ converges almost surely to $L_tf(x)$ for each fixed $t>0$ and fixed $x\in M,$ as the sample size $n\to\infty$. The above is a nonlocal Laplace-type operator with Gaussian kernel.

\begin{remark}\label{rmk:almost everywhere defined}
    Note that because $\mu$ is\textit{ non-atomic},  $\mu(\{x\})=0,$ so we have
    \[
    L_t f(x)
= \frac{1}{t^{d/2+1}} \int_{M\setminus \{x\}}
  e^{-\frac{d(x,y)^2}{t}}\big(f(x)-f(y)\big) p(y)\,d\mu(y).
    \]

\nd This is used later in Definition/Remark \ref{def:GL for loc. ang. conformal metrics}.
\end{remark}

\paragraph{\textbf{Intrinsic and extrinsic unnormalized graph Laplacians:}}
When we work intrinsically, the underlying space $M$ is viewed through its Riemannian structure, and the kernel is built from the intrinsic distance $d(\cdot,\cdot)$, namely the geodesic distance induced by the Riemannian metric. In many situations, however, the data are observed only in the ambient Euclidean space $\Ddim$. In that case, the same underlying set $M\subset \Ddim$ may instead be viewed extrinsically as a metric subspace of $\Ddim$, endowed with the ambient Euclidean distance rather than its intrinsic Riemannian distance. We then work with
\[
d(x,y):=\norm{x-y}_{\Ddim},
\]
and the Gaussian kernel is correspondingly denoted by
\[
k_t^{\operatorname{ext}}(x,y) = \exp\!\left(-\frac{\norm{x-y}_{\Ddim}^2}{t}\right).
\]
\noindent
The definitions of the discrete and continuous unnormalized graph Laplacians then change accordingly.

\begin{remark}[\textbf{Abuse of notation}]
    Note that we use $d$ both for the given metric on $M$ and for the Hausdorff dimension of the underlying metric space $(M,d)$. However, in most of this paper the distance is induced by a Riemannian metric $g$, and we write $d_g$ for this distance; hence no confusion should occur.
\end{remark}

\begin{remark}[\textbf{Dropping the terms '\textit{unnormalized}' and focusing only on the continuous graph Laplacian}]
    From now on and for the rest of this paper, we will only work with the \textit{continuous} graph Laplacian and we will drop the term \textit{'unnormalized'} and call it simply the \textit{graph Laplacian}. Therefore, unless otherwise specified, we will \textit{always} mean by graph Laplacian the continuous and unnormalized graph Laplacian $L_t$. Indeed, there are other types of graph Laplacian that are named \textit{random walk} and \textit{normalized} graph Laplacians whose asymptotics on Riemannian manifolds without boundaries were treated in \citet{HeinAudibertLuxburgJMLR2007}, but these are not the focus of this paper. Further when we deal with the locally angular conformal metric $g$ to a given metric $\tilde{g},$ we shall use 
    
\[
L_t^{g, int} f(x)
= \frac{1}{t^{d/2+1}}
  \int_M e^{-d_g(x,y)^2/t}\bigl(f(x)-f(y)\bigr)p(y)\,d\mathrm{vol}_g(y), 
  \,
L_t^{g, ext} f(x)
= \frac{1}{t^{d/2+1}}
  \int_M e^{-\frac{\norm{x-y}^2_{\Ddim}}{t}}\bigl(f(x)-f(y)\bigr)p(y)\,d\mathrm{vol}_g(y).  
\]

\nd to emphasize the intrinsic and extrinsic operators with respect to $g$.
\end{remark}

\paragraph{\textbf{Asymptotic behavior of graph Laplacian for manifolds without boundaries.}}
The small-bandwidth limit $t \to 0$ of~(2) determines how $L_t$ approximates the smooth and Riemannian structure of $M$.
These asymptotics matter for spectral methods used in geometric recovery, e.g. spectra-based manifold learning.
For smooth manifolds without a boundary, \citet{HeinAudibertLuxburgJMLR2007} (see Theorem 30), \citet{HeinAudibertLuxburgNIPS2005} (see Theorem 1), and the related work of
\cite{BelkinNiyogi2008} (see Theorem 3.3, and the derivation on P. 1301) showed that, for sufficiently regular $f$ and $p$,
\[
L_t f(x)
=
-\frac{\pi^{d/2}}{2}
\Big(
\tfrac12\,p(x)\,\Delta_g f(x)
+
\langle \nabla p(x), \nabla f(x) \rangle_g
\Big)
+ o(1),
\qquad
t \to 0,
\]
where $\Delta_g$ is the Laplace--Beltrami operator.
In particular, when $M$ is compact without boundary and $p$ is uniform,
$L_t$ pointwise approximates $-\tfrac{\pi^{d/2}}{4\,\mathrm{vol}_g M}\,\Delta_g$ as $t \to 0$. Note that compactness is only used for the special uniform-density simplification since $M$ must have finite volume to define uniform density, but compactness is not required for the general pointwise asymptotic results.
\nd We also note that the Diffusion Maps normalization of \citet{CoifmanLafon2006} recovers the Laplace–Beltrami operator independently of the sampling density, under the usual smooth compact-manifold assumptions.

\paragraph{\textbf{Extension to manifolds with smooth and non-smooth boundaries.}}

\nd The results for boundaryless manifolds have also been extended to manifolds with smooth and non-smooth boundaries. Near a smooth or non-smooth boundary, the graph Laplacian normally acquires a first-order term governed by the local inward sector; at interior points this term vanishes by symmetry and one recovers the density-weighted Laplace–Beltrami operator. Related smooth boundary-focused work includes \citet{BerrySauer2017BoundaryDensity} on density estimation on manifolds with boundary, and \citet{VaughnBerryAntil2024DiffusionBoundary} on diffusion maps for embedded manifolds with boundary and applications to PDEs. See also \citet{CalderParkSlepcev2022BoundaryEstimation} for boundary estimation from point clouds, and \citet{JiangHarlim2023GhostPointDM} for a ghost-point diffusion-maps approach to elliptic PDEs with classical boundary conditions.


In the setting of manifolds with kinks introduced in \citet{PalTewodrose2025}, which are manifolds with certain types of non-smooth boundaries that allow smooth boundaries as special cases, the same definition \eqref{eq:discrete_GL} applies. For a suitable definition of the Riemannian metric and its metric distance on manifolds with kinks, see \citet{PalTewodrose2025}.  
Denoting by $I_xM$ the inward tangent sector at $x\in M$ which is a generalization of the Bouligand tangent cone studied in convex analysis, and by $S^g_{I_xM}$ the corresponding inward tangent sphere w.r.t. $g$, Theorem 1.1 in \citet{PalTewodrose2025}  established that for intrinsic graph Laplacians:
\[
L_t f(x)
= -\frac{C_d}{\sqrt{t}}
   \big(p(x)\,B_{v_g(x)}f(x)+o(1)\big)
   -C_{d+1}\big(p(x)A_g f(x) + r(p,f)_g(x)\big)
   + O(\sqrt{t}),
   \qquad t\to0,
\]
where the term $o(1)$ becomes zero for interior points, $B_{v_g(x)}f$ denotes the derivative of $f$ in the \textit{generalized normal} direction $v_g(x)$ given by the \textit{mean direction} $v_g(x):=\int_{S^g_{I_xM}} \theta d\sigma(\theta),$ where $\sigma$ denotes the Hausdorff measure on the unit sphere $S^g_{I_xM}$ induced by $g,$ and 
\[
A_g f(x) = \frac{1}{2}\int_{S^g_{I_xM}}
   \operatorname{Hess}_{\tilde{g}}  f(x)(\theta,\theta)\,d\sigma(\theta),
   \qquad
r(p,f)_g(x) = \int_{S^g_{I_xM}}
   \langle\nabla f(x),\theta\rangle
   \langle\nabla p(x),\theta\rangle\,d\sigma(\theta),
\]
and $C_d:=\frac{1}{2}\Gamma(\frac{d+1}{2})$ is an explicit, positive constant depending only on the dimension $d \in \bbn$.  
Thus, boundary points, corner points, and more generally \textit{kinks} yield additional boundary terms proportional to $t^{-1/2}$ involving the \textit{generalized normal derivative} $B_{v_g(x)}f(x)$, namely the directional derivative along $v_g$. These additional terms vanish if one works with $f$ satisfying the \textit{generalized Neumann condition} $B_{v_g(x)}f(x)=0$. \citet{PalTewodrose2025} also proved a similar result for the extrinsic graph Laplacian, see Theorem 1.2 there.

\begin{remark}[\textbf{Identification of constants at interior points}]
For $x$ an interior point, $I_xM = T_xM$ above and hence $S^g_{I_xM} \cong \mathbb{S}^{d-1}$. By spherical symmetry,
\(
A_g f(x)
= \frac12 \frac{\mathrm{vol}(\mathbb{S}^{d-1})}{d}\,\Delta_g f(x),
r(p,f)_g(x)
= \frac{\mathrm{vol}(\mathbb{S}^{d-1})}{d}\,
\langle \nabla p(x), \nabla f(x) \rangle_g .
\)
Using the identity $C_{d+1} = \frac{d}{2} C_{d-1}$ together with
$\mathrm{vol}(\mathbb{S}^{d-1}) C_{d-1} = \pi^{d/2}$, we obtain
\(
C_{d+1}\frac{\mathrm{vol}(\mathbb{S}^{d-1})}{d}
= \frac{\pi^{d/2}}{2}.
\)
Consequently, the interior-point asymptotic in
\citet{PalTewodrose2025} reduces exactly to the
 constant in \citet{BelkinNiyogi2008}
\(
L_t f(x)
=
-\frac{\pi^{d/2}}{2}
\Big(
\tfrac12\,p(x)\,\Delta_g f(x)
+
\langle \nabla p(x), \nabla f(x) \rangle_g
\Big)
+ o(1),
\, t \to 0 .
\)
\end{remark}

\nd These results show that the unnormalized continuous graph Laplacian approximates the Laplace--Beltrami operator on smooth (meaning $\cC^{\infty}$) manifolds without boundary, and has natural extensions involving first-order operators on singular spaces such as manifolds with kinks. From the contrasting asymptotic behavior at the interior or boundary points, one may have the intuition that near \textit{singular} points (see Definition \ref{isolated/point singularity}), the graph Laplacian's asymptotic behavior is drastically different from non-singular points. In this paper, we investigate a specific type of singularity different from boundary or corner points and examine the behavior of the graph Laplacian at it; we notice that the behavior above is not\textit{ always} drastically different, but there are cases where it is so. The theorems in this study precisely illustrate these similarities and differences.

\subsection{Singularities in question: \textit{isolated} singularities}

As a general rule in this paper, by a \textit{singularity} of an incomplete metric space $M$ we mean a point in the metric completion $\widetilde{M}$ having a neighborhood whose intersection with $M$ is homeomorphic to a punctured Euclidean ball. We call this an \textit{isolated} or \textit{point} singularity (Definition \ref{isolated/point singularity}), following \citet{SmithYang1992}.

\subsubsection{\textbf{Asymptotics of graph Laplacian at isolated singularities}}

\citet{SmithYang1992} showed that the smooth (meaning $\cC^{\infty}$) manifold structure, Riemannian metrics, and exponential maps all extend across an isolated singularity, provided that the curvature function has suitably controlled blow-up (see Definition \ref{def:controlled-blow-up of curvature}); see Theorem \ref{extending smooth structure and metric for dim 3}. It is noteworthy that they worked with a more general class of Riemannian metrics, namely \textit{$\cC^1$ metrics with continuous Riemann curvature tensors}, whereas in this paper we restrict attention to smooth Riemannian metrics.\\\\

\nd One can ask about the asymptotic behavior of the graph Laplacians at that isolated singularity as $t\downarrow 0$. Below are our main theorems that answer this question:

\subsection{Main results}\label{scn:main-results}

\nd The main results of this study are divided along two distinct regimes:

\nd \textbf{Regime A (Riemannian metric extendable across isolated singularities).} Under the controlled curvature blow-up condition (Definition \ref{def:controlled-blow-up of curvature}) and Hölder regularity of the Ricci curvature (Theorem \ref{extending smooth structure and metric for dim 3}), the completion $\widetilde{M}$ carries a smooth structure and the metric extends across the singular point. In this regime, the isolated singularity becomes an interior point of $\widetilde{M}$ and the small-bandwidth limit of the graph Laplacian agrees with the classical interior-point expansion (Theorem \ref{thm:Asymptotics of Continuous graph Laplacian}).\\
\nd \textbf{Regime B (Riemannian metric non-extendable across isolated singularities via purely angular conformal change).} If one locally modifies a smooth background metric by a \textit{purely angular} conformal factor near the singularity (see Definition \ref{def:locally angularly conformal}, which is a new definition where the conformal factor is a function of the angle or direction\textit{ only}) and the factor is non-constant, then the metric does not extend across the singularity (Proposition \ref{prop:extendability}) and the curvature blow-up is necessarily non-integrable as in \citet{SmithYang1992}(Theorem \ref{thm:LAC_SY}). In this regime, the intrinsic and extrinsic graph Laplacians exhibit $t^{-1/2}$ growth rate under a mild non-degeneracy condition (Theorems  \ref{thm:blowup of GL near isolated singularity} and \ref{thm:TE extrinsic GL}). Below, we present the results for these two regimes.

\begin{theorem}\label{thm:Asymptotics of Continuous graph Laplacian}\textbf{(Asymptotics of the continuous graph Laplacian near isolated singularity of a manifold)}
    Let $M$ be a $d$-dimensional Riemannian manifold with an isolated singularity at $x.$\\
    \textbf{Case I:} Assume that $M$ has an intrinsic Riemannian metric $g,$ inducing a metric (distance) $d_g.$\\
    \textbf{Case II:} Assume that $M$ is isometrically embedded into $\Ddim,$ and thus, its Riemannian metric $g$ is induced from $\Ddim,$ and this induces a metric distance $d_g$ as in \textbf{ Case I}).\\
    
    \nd In both cases, $\widetilde{M}$ represents the metric completion of $M$ with respect to $d_g.$\\
      
    \nd Furthermore, in \textbf{Case II}, assume that if the metric completion $\widetilde{M}$ is a Riemannian manifold with the extended Riemannian metric denoted by $g$ again, then it is \textit{also} isometrically embedded into $\Ddim$ as a Riemannian submanifold. \label{embedding-assumption} (this is not always the case, see Remark \ref{rmk:explantion-embedding} below).\\

    \nd Next, in both cases above, assume the hypotheses about the singularity and henceforth, the conclusion, of Theorem \ref{extending smooth structure and metric for dim 3} when $d\ge 3,$ and the same for the Theorem \ref{extending smooth structure and metric for dim 2} when $d=2$. Then for both the \textit{intrinsic} (cf. \textbf{Case I}) above and \textit{extrinsic} graph Laplacians (cf.\textbf{ Case II}) at the isolated singularity $x$, both denoted by $L_tf(x),$ we have:

\[
L_t f(x)= -\,\frac{\pi^{d/2}}{2}\Big(\nabla_g f(x)\!\cdot\!\nabla_g p(x)\;+\;\tfrac12\,p(x)\,\Delta_g f(x)\Big) + O(\sqrt{t}),
\qquad \text{as } t\downarrow 0.
\]

\end{theorem}

\begin{remark}\label{rmk:explantion-embedding}
    The intrinsic statement requires no ambient embedding, whereas the extrinsic statement assumes an embedding of $\widetilde{M}.$
    In Theorem \ref{embedding-assumption}, we did \textit{not} assume that $\widetilde{M}$ is a smooth manifold. \citet{SmithYang1992} provided sufficient conditions (see Theorem \ref{extending smooth structure and metric for dim 3}) for when $\widetilde{M}$ carries a smooth structure with a $\cC^1$ Riemannian metric, but these conditions state nothing about $\widetilde{M}$ being embedded as a Riemannian \textit{submanifold} of $\Ddim$ even if $M$ is so. This is the point behind the seemingly strong extra embedding assumption in Theorem \ref{thm:Asymptotics of Continuous graph Laplacian}. We assume that \textit{whenever} $\widetilde{M}$ is a smooth Riemannian manifold, it is \textit{also} an isometrically embedded Riemannian submanifold of $\Ddim.$ A simple counterexample shows this is not automatic: the cone $M=\{(x,y,z)\in\mathbb{R}^3:x^2+y^2=z^2,\ z>0\}$ is a smooth embedded submanifold of $\mathbb{R}^3$, but its metric completion adds the apex, and the completed set, although an abstract manifold in its own right, is no longer a smooth Riemannian submanifold of $\mathbb{R}^3$, hence does not satisfy the  extra embedding assumption in Theorem \ref{thm:Asymptotics of Continuous graph Laplacian}. The extra assumption is what additionally allows us to regard it as an embedded Riemannian submanifold of $\Ddim$, and hence to define the \textit{extrinsic} graph Laplacian at $x.$ Without this assumption, we \textit{cannot} define the \textit{extrinsic} graph Laplacian at $x.$
\end{remark}

 \begin{proposition}[\textbf{Extendability of locally angularly conformal metrics $\Leftrightarrow$ constant angular conformal factor}]\label{prop:extendability}
Let $(\widetilde M,\widetilde g)$ be a smooth Riemannian manifold and let $x\in\widetilde M$. Let $g$ be locally angularly conformal to $\tilde{g}$ near $x$. The following statements are equivalent:
\begin{enumerate}
    \item The metric $g$ extends across $x$ as a continuous positive-definite symmetric bilinear form on $T_x\widetilde M$; equivalently, there exists an $R>0$ so that $g$ extends to a continuous Riemannian metric on the $\tilde{g}$-ball $B_{\widetilde g}(x;R)$.
    \item The angular conformal factor $\Psi_1\in \cC^2(U_x\widetilde{M})$ is constant on the $\tilde{g}$-unit sphere $U_x\widetilde M$ in the tangent space $T_x\widetilde M$.
\end{enumerate}
\end{proposition}

Note that the definition of 'locally angularly conformal metrics' is new, and it means that we modify the original Riemannian metric by a conformal factor which is a function of the direction or angle only, and not of the radius, see Definition \ref{def:locally angularly conformal}. In this case we rightly call it an \textit{angular conformal factor}. The following theorem states that the only way we can have a \textit{controlled curvature blow up} (cf. Definition \ref{def:controlled-blow-up of curvature}) in this case is when the angular conformal factor is\textit{ constant}. Denote by $\Pi\subset T_y\widetilde M$ a two-plane and $K_g(y,\Pi)$ its sectional curvature.

\begin{theorem}[\textbf{Locally angularly conformal metrics and controlled blow–up of curvature}]\label{thm:LAC_SY}
Let $g$ be a locally angularly conformal metric to $\tilde{g}$ with angular conformal factor $\Psi_1\in \cC^2(U_x\widetilde{M})$ .
For $0<s<\varepsilon<R$ set the curvature function (cf. Definition \ref{def:curvature function})
\[
\kappa_{\varepsilon}(s)\;:=\;\sup\bigl\{|K_g(y,\Pi)|:\ s<\widetilde d(x,y)<\varepsilon,\ \Pi\subset T_y\widetilde M\text{ a two–plane}\bigr\}.
\]
Then, for all dimensions $d\ge2$, the following hold.
\begin{enumerate}
\item[\emph{(i)}] If $\Psi_1$ is constant on the $\tilde{g}$-unit sphere $U_x\widetilde M$ in the tangent space $T_x\widetilde{M}$, then $\kappa_{\varepsilon}(s)=O(1)$ as $s\downarrow0$ and therefore $\displaystyle\int_0^\varepsilon s\,\kappa_{\varepsilon}(s)\,ds<\infty$.
\item[\emph{(ii)}] If $\Psi_1$ is not constant on $U_x\widetilde M$, then there exist $c,C,s_0>0$ such that
\[
\frac{c}{s^2}\ \le\ \kappa_{\varepsilon}(s)\ \le\ \frac{C}{s^2}\qquad(0<s<s_0),
\]
and consequently $\displaystyle\int_0^\varepsilon s\,\kappa_{\varepsilon}(s)\,ds=\infty$.
\end{enumerate}
In particular, the Smith–Yang integrability condition $\int_0^\varepsilon s\,\kappa_{\varepsilon}(s)\,ds<\infty$ holds if and only if $\Psi_1$ (equivalently, the angular conformal factor) is constant on $U_x\widetilde M$.
\end{theorem}

While the Proposition \ref{prop:extendability} and Theorem \ref{thm:LAC_SY} above give us an \textit{equivalent} condition of the extendability of the locally angularly conformal metric across an isolated singularity, it does \textit{not} address the asymptotic behavior of the graph Laplacian near $x$ for a \textit{non-extendable} locally angularly conformal metric, i.e. for a \textit{non-constant} angular conformal factor (cf. Proposition \ref{prop:extendability}). Before stating the theorem that addresses it, let us \textit{formally} define:

\begin{defremark}[\textbf{Graph Laplacian for locally angular conformal metrics}]\label{def:GL for loc. ang. conformal metrics}
    Noting that the locally angularly conformal metric $g$ may \textit{not} be defined at $x$ (cf. Proposition \ref{prop:extendability}), the volume form $vol_g$ may also not be defined at $x.$ See Remark \ref{rmk:almost everywhere defined}. Hence we \textit{extend} the measure given by $vol_g$ on $M\setminus \{x\}$ by giving zero mass to $x,$ (so that new measure on $M$ is still non-atomic) and \textit{still} call the resulting measure $vol_g.$ Hence we \textit{define} the graph Laplacian $L_t^gf(x)$ as:
    \[
     L_t^gf(x):= \frac{1}{t^{d/2+1}}\int_{M}e^{-\frac{d^2(x,y)}{t}}(f(x)-f(y))p(y)dvol_g(y):=  \frac{1}{t^{d/2+1}}\int_{M\setminus \{x\}}e^{-\frac{d^2(x,y)}{t}}(f(x)-f(y))p(y)dvol_g(y)
    \]
\end{defremark}

\nd Note that the last integral above makes perfect sense since $vol_g$ was defined on $M\setminus \{x\}.$ With that definition, the following theorem puts a \textit{mild} condition on the locally angularly conformal metrics to ensure blow up of the graph Laplacian at the isolated singularity as $t\downarrow 0$.\\

\begin{theorem}[\textbf{Taylor expansion of the intrinsic graph Laplacian near an isolated singularity for a non-constant locally angular conformal change}]\label{thm:blowup of GL near isolated singularity}
    Let $(\widetilde{M},\tilde{g})$ be a smooth Riemannian manifold and let $M:=\widetilde{M}\setminus\{x\}$. Let $g$ be a Riemannian metric on $M$ that is locally angularly conformal to $\tilde{g}$ near $x$, as in Definition \ref{def:locally angularly conformal}. Consider the $\tilde{g}$-unit sphere $U_x\widetilde{M}$ in $T_x\widetilde{M}$ with the Hausdorff measure $d\sigma(\Theta)$ induced by $\tilde{g}.$ Let $p$ be the density on $M$ w.r.t. $dvol_g, p(x)\neq 0$ that extends as a $\cC^2$ function on $\widetilde{M}$. Let $f\in C^3(\widetilde M)$ satisfy the conditions in the Proposition \ref{prop:truncation-error} with $\mu:=dvol_g.$     
\nd Let
$
y = \exp^{\tilde g}_x(r,\Theta), r\in[0,\rho),\ \Theta\in U_x\widetilde M,
$
be $\widetilde g$–geodesic normal coordinates at $x$. Assume that there exist $\rho>0, C>0, \delta >0$ and a positive continuous function 
$L:U_x\widetilde M\to(0,\infty)$ such that, for all $0<r<\rho$ and for all 
$\Theta\in U_x\widetilde M$,
\[
d_g\!\bigl(x,\exp_x^{\tilde g}(r\Theta)\bigr)^2
  = L(\Theta)^2 r^2 + E(r,\Theta), 
  \qquad |E(r,\Theta)|\le C r^{2+\delta}, 
\]
\nd Assume the following non-degeneracy conditions: 
    
\begin{equation}\label{eqn:condition-for-blowup}
p(x)\ne 0,\text{ and }
     \left\langle \nabla_{\tilde{g}}{f}(x),\int_{U_x\widetilde{M}}\Theta e^{\frac{d}{2}\Psi_1(\Theta)}L(\Theta)^{-(d+1)}d \sigma(\Theta) \right\rangle_{\tilde{g}} 
\ne 0.  
\end{equation}

\nd For $t>0$ define the intrinsic continuous graph Laplacian w.r.t. the metric $g$ by
\[
L_t^{g, int} f(x)
= \frac{1}{t^{d/2+1}}
  \int_M e^{-d_g(x,y)^2/t}\bigl(f(x)-f(y)\bigr)p(y)\,d\mathrm{vol}_g(y).
\]

\noindent For $k\ge0$ set $c_k := \frac12\Gamma\!\bigl(\frac{k+1}{2}\bigr)$.  
\noindent Define
\[
b(x) := \int_{U_x\widetilde M} 
       \Theta\, e^{\frac d2\Psi_1(\Theta)}\,L(\Theta)^{-(d+1)}\,d\sigma(\Theta)
       \in T_x\widetilde M,
\Phi(\Theta)
:= \frac12\,p(x)\,\bigl(\nabla^2_{\tilde{g}} f(x)\bigr)(\Theta,\Theta)
  + \langle\nabla_{\tilde{g}} p(x),\Theta\rangle\langle\nabla_{\tilde{g}} f(x),\Theta\rangle,
\]
and
\[
B_0(x)
:= \int_{U_x\widetilde M}
     \Phi(\Theta)\, e^{\frac d2\Psi_1(\Theta)}\, L(\Theta)^{-(d+2)}\, d\sigma(\Theta).
\]

\nd Then, as $t\downarrow 0$, for every $\delta>0,$

\[\sqrt{t}L_t^{g, int}f(x)=O(1)\]

\nd Furthermore for $\delta \ge 2,$

\[
L_t^{g, int} f(x)
= -\,c_d\,p(x)\,\langle\nabla_{\tilde{g}} f(x),b(x)\rangle_{\tilde{g}}\,t^{-1/2}
  - c_{d+1}\,B_0(x)
  + O(\sqrt{t}).
\]

\nd In particular, if $\langle\nabla f(x),b(x)\rangle\neq 0$, and $\delta \ge 2$ then $|L_t^{g, int} f(x)|\to\infty$ with rate $t^{-1/2}$.
\end{theorem}

\begin{remark}[\textbf{Interpretation of $L(\Theta)$}] Note that 
\[ L(\Theta)=\lim_{r\to 0}\frac{d_g\bigl(x,\exp_x^{\tilde g}(r\Theta)\bigr)}{r}. \] The function $L(\Theta)$ can be interpreted as the \textit{limiting distortion factor} of $\tilde{g}$-distance along the direction $\Theta$ after the angular conformal change to $g.$
    
\end{remark}

\begin{remark}The reason the $\tilde{g}$-gradient $\nabla_{\tilde{g}}$ appears in the results of Theorems \ref{thm:blowup of GL near isolated singularity}, \ref{thm:TE extrinsic GL} is that all calculations are performed in $\tilde{g}$ coordinates around $x,$ since $g$-coordinates are not defined at $x$, cf. Proposition \ref{prop:extendability}.
    
\end{remark}

\begin{remark}
    It is clear that if $\Psi_1\equiv c$, then $L(\Theta)\equiv e^{c/2}$ and $E(r,\Theta)=0$. Then
    \[
    b(x):=\int_{U_x\widetilde{M}}\Theta \, e^{\frac{d}{2}\Psi_1(\Theta)}L(\Theta)^{-(d+1)}\, d \sigma(\Theta) =0
    \]
    by symmetry, hence the assumption
    \[
    \left\langle \nabla_{\tilde{g}}f(x),\int_{U_x\widetilde{M}}\Theta e^{\frac{d}{2}\Psi_1(\Theta)}L(\Theta)^{-(d+1)}\, d \sigma(\Theta) \right\rangle_{\tilde{g}} \ne 0
    \]
    in Theorem \ref{thm:blowup of GL near isolated singularity} does not hold. In this case, $g$ extends across $x$ by Proposition \ref{prop:extendability}. The limiting second-order term is then obtained by applying the interior-point graph Laplacian asymptotic to the extended metric, and is therefore consistent with the boundaryless case discussed earlier.
\end{remark}

\nd Next is the theorem that corresponds to the \textit{extrinsic} case. 

\begin{definition}[\textbf{Extrinsic kernel continuous graph Laplacian}]\label{def:ext-kernel-GL}
    \nd For $t>0$ define the extrinsic kernel continuous graph Laplacian (see Remark \ref{rmk:clarification-extrinsic}) w.r.t. the metric $g$ by
\[
L_t^{g, ext} f(x)
= \frac{1}{t^{d/2+1}}
  \int_M e^{-\frac{\norm{x-y}^2_{\Ddim}}{t}}\bigl(f(x)-f(y)\bigr)p(y)\,d\mathrm{vol}_g(y).
\]
\end{definition}

\begin{remark}[\textbf{On the meaning of ``extrinsic'' in Definition \ref{def:ext-kernel-GL} and Theorem \ref{thm:TE extrinsic GL}}]\label{rmk:clarification-extrinsic}
In Definition \ref{def:ext-kernel-GL} and Theorem \ref{thm:TE extrinsic GL}, the Gaussian kernel is defined using the ambient Euclidean distance
$\|x-y\|$ coming from a fixed isometric embedding of $(\tilde M,\tilde g)$ into
$\mathbb{R}^D$. The operator $L^{g,\mathrm{ext}}_t$ is therefore \emph{extrinsic in its
kernel} (ambient distance), while the weighting is given by the $g$--volume form
$d\mathrm{vol}_g$ and the density $p$ w.r.t. $vol_g$. In particular, after the conformal change, we do \emph{not} assume that $(M,g)$ itself is isometrically embedded in $\mathbb{R}^D$ as a Riemannian submanifold; rather, we only use that the underlying set $M$ sits inside $\mathbb{R}^D$ as a metric subspace with the ambient Euclidean distance, so that the kernel is built from $\|x-y\|$, not from the intrinsic Riemannian distance associated with $g$. This extrinsic-kernel/$g$-weighted form is standard in the graph Laplacian literature, where kernels are
constructed from ambient distances while the limiting operator depends on the
underlying Riemannian structure and sampling density; see, for instance,
\citet{BelkinNiyogi2008, HeinAudibertLuxburgJMLR2007} and references therein.
\end{remark}

\begin{theorem}[\textbf{Taylor expansion of the extrinsic kernel graph Laplacian at an isolated singularity for a non-constant locally angular conformal change}]\label{thm:TE extrinsic GL}
    We follow the setup and notations of Theorem \ref{thm:blowup of GL near isolated singularity}, with $g$ locally angularly conformal to $\tilde{g}$. Assume that $(\widetilde{M}, \tilde{g})$ is isometrically embedded in $\Ddim.$ Denote by $L_t^{g, ext}f$ the extrinsic kernel graph Laplacian of $f$ with respect to $g.$

\nd Denote by 
    \begin{align*}
c_d &:= \int_0^\infty e^{-r^2}r^{d}\,dr,
c_{d+1}:= \int_0^\infty e^{-r^2}r^{d+1}\,dr,\\
B_M f(x)
&:= \int_{U_x\widetilde M} e^{\frac{d}{2}\Psi_1(\Theta)}
    \,\langle\nabla_{\tilde{g}} f(x),\Theta\rangle\,d\sigma(\Theta),\\
A_M f(x)
&:= \frac{1}{2}\int_{U_x\widetilde M} e^{\frac{d}{2}\Psi_1(\Theta)}
    \,\nabla^2_{\tilde{g}}{f}(x) (\Theta,\Theta)\,d\sigma(\Theta),\\
    r(p,f)_M(x)&:= \int_{U_x\widetilde M} e^{\frac{d}{2}\Psi_1(\Theta)}
    \,\langle\nabla_{\tilde{g}}  f(x),\Theta\rangle
     \,\langle\nabla_{\tilde{g}}  p(x),\Theta\rangle\,d\sigma(\Theta).
\end{align*}
Then $$L_t^{g, ext}f(x)= (1+o(1))\left(-\frac{c_d}{\sqrt{t}}\,p(x)\,B_M f(x)
   \;-\;c_{d+1}\Bigl(p(x)\,A_M f(x)
                    +r(p,f)_M(x)\Bigr)\right)
   +O(t^{1/2}), t\downarrow 0.$$
\end{theorem}

\begin{remark}
In the extrinsic case the kernel involves the ambient Euclidean distance
$\|x-y\|_{\mathbb{R}^D}$. Writing $u=\pi_x(y)-x\in T_x\widetilde M \simeq \mathbb{R}^d$ where $\pi_x$ is the orthogonal projection of  $\Ddim$ to the translated tangent space $x+T_x\widetilde{M}$, its Taylor expansion takes the form
\[
\|x-y\|_{\mathbb{R}^D}^2
  = \|u\|^2 + Q_{x,4}(u) + Q_{x,5}(u) + O(\|u\|^6),
\]
where $Q_{x,m}$ denotes a homogeneous polynomial of degree $m$. See \citet{CoifmanLafon2006}, Appendix B, Lemma 7, or \citet{PalTewodrose2025}, Lemma 5.1.

In the intrinsic case the kernel depends on $d_g(x,y)$, where $g=e^{\Psi_1(\Theta)}\tilde g$ is only angularly conformal near the singularity. Since Definition \ref{def:locally angularly conformal} is local near $x$, no claim is made away from $x$.
For such metrics an expansion of the form
\[
d_g\bigl(x,\exp_x^{\tilde g}(r,\Theta)\bigr)^2
  = L(\Theta)^2 r^2 + O(r^{2+\delta}), \delta > 0
\]
is not automatic; this hypothesis is precisely what allows us to isolate the
anisotropic factor $L(\Theta)$, and furthermore 
derive a full Taylor expansion with remainder
$O(\sqrt t)$ when $\delta \ge 2.$ 
\end{remark}

\begin{remark}[\textbf{On the relation between Theorems \ref{thm:Asymptotics of Continuous graph Laplacian}, \ref{thm:blowup of GL near isolated singularity}, and \ref{thm:TE extrinsic GL}}]
The three theorems describe \emph{different geometric regimes} near an isolated singularity, and this is precisely why they lead to different asymptotic behaviors.

Theorem \ref{thm:Asymptotics of Continuous graph Laplacian} treats the \emph{removable} regime: under the Smith--Yang hypotheses, the singularity can be filled in so that the metric extends across $x$, and $x$ becomes an interior point of the completed manifold $\widetilde M$. The problem is then reduced to the already known interior-point asymptotics of the graph Laplacian, which explains both the bounded limit and the shorter proof, as we invoked the already proven asymptotics for the graph Laplacian.

By contrast, Theorems \ref{thm:blowup of GL near isolated singularity} and \ref{thm:TE extrinsic GL} treat a \emph{non-removable} regime produced by a purely angular conformal change near $x$. In this setting, the metric typically does not extend across the singularity (cf. Proposition \ref{prop:extendability}), and the graph Laplacian no longer reduces to the standard interior-point case. Instead, one obtains a leading term of order $t^{-1/2}$ under a non-degeneracy condition.

The difference between the latter two theorems is only in the choice of kernel distance: Theorem \ref{thm:blowup of GL near isolated singularity} is \emph{intrinsic}, using the distance $d_g$, whereas Theorem \ref{thm:TE extrinsic GL} is \emph{extrinsic}, using the ambient Euclidean distance. Thus, the three theorems are complementary rather than competing: Theorem \ref{thm:Asymptotics of Continuous graph Laplacian} covers the removable case, while Theorems \ref{thm:blowup of GL near isolated singularity} and \ref{thm:TE extrinsic GL} cover a specific \textit{non-removable} conformal regime, in intrinsic and extrinsic form respectively.
\end{remark}


\begin{remark}[\textbf{Detection of isolated singularities from data}]
\noindent
The asymptotic behavior in Theorem \ref{thm:TE extrinsic GL} also suggests a mechanism for detecting singularities from sampled data. Since the discrete extrinsic graph Laplacian is an empirical approximation of its continuous counterpart, one expects that for sufficiently small bandwidths $t$, the observed behavior of $L_{n,t}^{g,\mathrm{ext}}f(x)$ reflects that of $L_t^{g,\mathrm{ext}}f(x)$. In particular, in the non-removable angular conformal regime, the characteristic $t^{-1/2}$ blow-up of $L_t^{g,\mathrm{ext}}f(x)$ yields a distinctive signature that is absent in the removable regime (e.g. Theorem \ref{thm:Asymptotics of Continuous graph Laplacian}), where the operator remains bounded. This suggests that singularities may be identified in practice by examining the scaling behavior of graph Laplacian estimators across bandwidths. Such a criterion is consistent with the interpretation of first-order asymmetry quantities, such as the generalized normal term, as indicators of non-smooth geometric structure.
\end{remark}

We do not pursue finite-sample guarantees, spectral convergence, eigenfunction estimation, or kernel-interpolation consequences in this paper. These questions require additional control of the empirical operator \(L_{n,t}\) near the singularity and are left for future work. In particular, it would be interesting to understand whether the \(t^{-1/2}\)-scaling signature can be used to design stable renormalizations for spectral or kernel-based methods near non-removable singularities.


\subsection{Organization of this paper}

\nd The paper is organized as follows. In the next section \ref{scn:extension of smooth structure}, we discuss the setup for our isolated singularity following \citet{SmithYang1992}, and also provide a proof for Theorem \ref{thm:Asymptotics of Continuous graph Laplacian}. In Section \ref{scn:non-extn}, we introduce locally angularly conformal metrics, and examine the behavior of the graph Laplacian, and prove Proposition \ref{prop:extendability}, Theorems \ref{thm:blowup of GL near isolated singularity}, \ref{thm:TE extrinsic GL}. 
Finally, Section \ref{scn:simulations} provides some relevant simulations.

\subsection{Acknowledgements} This research was funded by the Research Foundation–Flanders (FWO) via the Odysseus II programme no. G0DBZ23N. I am grateful to David Tewodrose, my postdoctoral supervisor, and Manuel Dias, my colleague, for taking their time to read the manuscript and offering helpful suggestions to improve the same.

\section{Extending smooth structure at isolated singularities}\label{scn:extension of smooth structure}

\subsection{Definitions and theorems}

\nd Let $(M,g)$ be a smooth Riemannian manifold with a smooth Riemannian metric $g$ that is not geodesically complete. Then the Riemannian metric $g$ induces a distance function on $M:d_g(x,y):=inf_{\gamma} l_g[\gamma]$, where $\gamma$ is a $\mathcal{C}^1$ curve from $x$ to $y,$ and $l_g[\gamma]$ its length with respect to $g.$ Let $\widetilde{M}$ be the completion of $M$ w.r.t. $d_g$ as a metric space. 

\begin{definition}\label{isolated/point singularity}\textbf{(isolated singularity)}
    \nd A point $p\in \widetilde{M}$ is called an \textit{isolated singularity} of $M$ if there is a neighborhood $N\subset \widetilde{M}$ of $p$ so that $N\cap M=N\setminus \{p\}.$
\end{definition}

\begin{definition}\label{def:curvature function}\textbf{(curvature function near a point singularity)}
Let $p\in \widetilde M$ be an isolated singularity of $M$. We consider the Levi-Civita connection of $g$ on $M$. For $x\in M$ and a two-plane $E\subset T_xM$, denote by $K_g(x,E)$ the sectional curvature of $E$. Fix $\epsilon>0$ small enough so that no other isolated singularity lies in the punctured ball \(B_g(p;\epsilon)\setminus\{p\}\). Define
\[
|K|:M\to[0,\infty),\qquad
|K|(x):=\sup_{\substack{E\subset T_xM\\ \dim E=2}} |K_g(x,E)|,
\]
and
\[
\kappa_\epsilon:(0,\epsilon)\to[0,\infty),\qquad
\kappa_\epsilon(s):=
\sup_{\{x\in M:\, s<d_g(p,x)<\epsilon\}} |K|(x).
\]
Following \citet{SmithYang1992}, we call $\kappa_\epsilon(s)$ the curvature function of $g$ at the isolated singularity $p$.
\end{definition}

\begin{definition}[\textbf{controlled blow-up of curvature}]\label{def:controlled-blow-up of curvature}
    Following \citet{SmithYang1992}, we say that $(M,g)$ has \textit{controlled blow up of curvature} if $\exists \epsilon >0$ so that $\int_{0}^{\epsilon}s\kappa_{\varepsilon}(s)ds<\infty.$
\end{definition}

\nd Next we state the relevant parts of the two main theorems from \citep{SmithYang1992} for $dim\ge 3$ and for $dim=2$ that deal with extending the smooth structure and Riemannian metric across an isolated singularity.

\begin{theorem}\label{extending smooth structure and metric for dim 3}\textbf{(extending smooth structure and Riemannian metric in dim $\ge 3:$ Theorem 3.1 in \citep{SmithYang1992})}

\nd For $d\ge 3,$ assume that $(M,g)$ is a smooth $d$-dimensional geodesically incomplete Riemannian manifold with a smooth Riemannian metric $g$, and with an isolated singularity at $p.$ Denote by $\widetilde{M}$ its completion as a metric space w.r.t. the metric $d_g$, the metric induced on $M$ by the Riemannian metric $g$. Assume furthermore that there exists an $\epsilon >0$ so that:
 
\begin{itemize}
    \item The punctured ball $B(p;\epsilon)\setminus \{p\}$ is simply connected.
    \item $\int_{0}^{\epsilon}s\kappa_{\varepsilon}(s)ds < \infty.$ (cf: Definition \ref{def:controlled-blow-up of curvature})
    \item There is no geodesic loop in $B(p;\epsilon)\setminus \{p\}$ with both endpoints at $p.$
    \item The Ricci curvature is $C^{0,\alpha}, 0 < \alpha < 1.$
\end{itemize}

\nd Then:
\begin{itemize}
    \item There is a neighborhood $N\subset \widetilde{M}$ around $p$ so that $N$ has a smooth manifold structure compatible with the original smooth structure on \(N\cap M\).
    \item $g$ extends to a $\cC^{2,\alpha}, 0 < \alpha < 1$ Riemannian metric $\tilde{g}$ on $N.$
    \item The exponential map $\exp^{\tilde g}_p$ is a local diffeomorphism from a neighborhood of $0\in T_p\widetilde M$ onto $N$, and is a local radial isometry, i.e.
\[
d_{\tilde g}\bigl(p,\exp^{\tilde g}_p(r\theta)\bigr)=r
\]
for all sufficiently small $r>0$ and all unit vectors $\theta\in T_p\widetilde M$.
\end{itemize}
  
\end{theorem}

\begin{theorem}\label{extending smooth structure and metric for dim 2}\textbf{(extending smooth structure and Riemannian metric in dim $=2:$ Theorem 3.3 in \citep{SmithYang1992})}

\nd Assume that $(M,g)$ is a $2$-dimensional geodesically incomplete Riemannian manifold with a smooth Riemannian metric $g$, and with an isolated singularity at $p.$ Denote by $\widetilde{M}$ its completion as a metric space with respect to the metric $d_g$ induced by the Riemannian metric $g$. Assume furthermore that there exists an $\epsilon >0$ so that:\\

\begin{itemize}
    \item The closure of the ball $B(p;\epsilon)$ is compact in the topology of $\widetilde{M}$.
    \item $\int_{0}^{\epsilon}s\kappa_{\varepsilon}(s)ds < \infty.$  (cf: Definition \ref{def:controlled-blow-up of curvature})
    \item The Gauss-Bonnet theorem holds for $B(p;\epsilon),$ i.e. given any open subset $D\subset B(p;\epsilon)$ with smooth boundary $\partial D \subset B(p;\epsilon)\setminus \{p\}\subset M,$

    $$\int_D GdA + \int_{\partial{D}} gds = 2\pi \chi(D).$$

    \nd where $G$ is the Gauss curvature, $g$ is the curvature of the curve $\partial D$, and $\chi(D)$ is the Euler characteristic of $D.$

\end{itemize}

\nd Then the conclusion of the previous Theorem \ref{extending smooth structure and metric for dim 3} holds.
    
\end{theorem}

\begin{remark}
    Note that in Theorem \ref{extending smooth structure and metric for dim 2} above, the conclusion part for $d=2$ case indeed means that the Riemannian metric extends as a $\cC^{2, \alpha}, 0 < \alpha < 1$ metric across the singularity. The proof in Section 11 of their paper, using almost linear and harmonic coordinates, holds in \textit{all} dimensions without change. This fact has also been verified directly by the author with Prof. Deane Yang, one of the authors in \citet{SmithYang1992}, by private communication.
\end{remark}

\begin{remark}
    Note that we added the statement on the exponential map being a local diffeomorphism and a radial isometry around $x$ in the conclusion part of the above two theorems, that were not \textit{explicitly} mentioned in the original paper, but on P.216, Section 10 of \citet{SmithYang1992}, they construct \textit{geodesic normal coordinates} at the singularity that correspond to the inverse of the exponential map which is shown to be a $\cC^1$ diffeomorphism, and this guarantees in Theorem \ref{extending smooth structure and metric for dim 3} that the exponential map for $\tilde{g}$ is a local diffeomorphism and a local radial isometry near $p$ . 
\end{remark}

\subsection{Proof of Theorem \ref{thm:Asymptotics of Continuous graph Laplacian}}

\begin{proof}
  The proof of Theorem \ref{thm:Asymptotics of Continuous graph Laplacian} follows at once when we apply the Theorems \ref{extending smooth structure and metric for dim 3}, \ref{extending smooth structure and metric for dim 2} that allow us to extend the Riemannian metric $g$ across the isolated singularity $x$ as a $\cC^{2,\alpha}, 0< \alpha < 1$ metric across $x, x$ becomes an \textit{interior} point of $\widetilde{M},$ and so the limit follows from the results of graph Laplacian at interior points proved in  \citet{BelkinNiyogiCOLT2005} (Theorem 5.1 and its proof), \citet{HeinAudibertLuxburgNIPS2005} (Theorem 1). Finally, Theorem 1.1, 1.2 in \citet{PalTewodrose2025} assume the Riemannian metric is $\cC^2$, and treat interior points as a special case, and from there, the error of the order $O(\sqrt{t})$ is justified. 
\end{proof}

\section{Non-extension of Locally Angularly Conformal Riemannian metrics across singularity and blow-up of graph Laplacians} \label{scn:non-extn}

\subsection{Locally Angularly Conformal Metrics}

We now work in a different local setup from Section \ref{scn:extension of smooth structure}: here $(\widetilde M,\tilde g)$ is given a priori as a smooth Riemannian manifold, $x\in\widetilde M$ is fixed, and $M:=\widetilde M\setminus\{x\}$. 

\begin{definition}[\textbf{Riemannian logarithm}]\label{def:log-map}
In the range of the Riemannian exponential map $\exp_x^{\tilde g}$, we denote by $\log_x^{\tilde g}$ its inverse.
\end{definition}

\begin{definition}[\textbf{local angular conformality of Riemannian metrics}]\label{def:locally angularly conformal}
Denote by $U_x\widetilde{M}$ the $\tilde{g}$-unit sphere in the tangent space of $\widetilde{M}$. Let $g$ be another Riemannian metric on $M.$ We say that $g$ is \textit{locally angularly conformal} on $M$ to $\tilde{g}$ \textit{near $x$ }if there exists a $\tilde{g}$-ball $B_{\tilde{g}}(x;R)\subset \widetilde{M}$ so that on the \textit{punctured ball} $B_{\tilde{g}}(x;R)\setminus \{x\},$

\[g(y)=e^{\Psi(y)}\tilde{g}(y)\]

\nd for some $\cC^2$ function $\Psi:B_{\tilde g}(x;R)\setminus \{x\}\to \cl$ such that there exists a function $\Psi_1:U_x\widetilde M\to \cl$ satisfying
\[
\Psi(y)=\Psi_1\!\left(\frac{\log_x^{\tilde g} y}{d_{\tilde g}(x,y)} \right)
= \Psi_1\!\left(\frac{\log_x^{\tilde g} y}{\|\log_x^{\tilde g} y\|_{\tilde g}} \right)
\quad \text{for each } y\in B_{\tilde g}(x;R)\setminus \{x\}.
\]

\nd We call $\Psi_1$ the \textit{angular conformal factor}.
The above definition means for time $t\le R, \Psi$ is constant along unit speed $\tilde{g}$-radial geodesics emanating from $x$.
    
\end{definition}
\begin{remark}\label{rmk:same-completion}
  Although \(g\) need not extend as a Riemannian metric tensor across \(x\), the boundedness of the angular conformal factor implies that \(d_g\) and \(d_{\tilde g}\) are locally bi-Lipschitz equivalent near \(x\). Hence the metric completion of \(M=\widetilde M\setminus\{x\}\) with respect to \(d_g\) is locally identified with \(\widetilde M\), obtained by adding back the point \(x\).  
\end{remark}

\subsection{Proof of Proposition \ref{prop:extendability}}

\begin{proof}
Let $(\widetilde M,\widetilde g)$ be a smooth Riemannian manifold and $x\in\widetilde M$. 
Assume that on a punctured neighborhood $U\setminus\{x\}$ one has
\[
g=e^{\Psi_1(\Theta)}\,\widetilde g,
\]
where $\Theta\in U_x\widetilde M$ is the 'angular variable' (in $\widetilde g$–polar coordinates around $x$) and $\Psi_1:U_x\widetilde M\to \cl$ is continuous.

\smallskip
\noindent\emph{($\Rightarrow$)} Suppose $\Psi_1\equiv c\in \cl $ on $U_x\widetilde M$. Then $g\equiv e^c\,\widetilde g$ on $U\setminus\{x\}$. 
Define $g_x:=e^c\,\widetilde g_x$; this is a positive–definite symmetric bilinear form on $T_x\widetilde M$. 
Since $\widetilde g$ is continuous, the tensor field $g$ extends continuously to $x$ by $g_x$.

\smallskip
\noindent\emph{($\Leftarrow$)} Conversely, assume there exists a continuous $(0,2)$–tensor field $G$ on $U$ such that
$G=g$ on $U\setminus\{x\}$. Fix $\Theta\in U_x\widetilde M$ and consider the $\widetilde g$–geodesic ray
$\gamma_\Theta(r)=\exp_x^{\tilde g}(r\Theta)$ for $r>0$ sufficiently small. On $U\setminus\{x\}$ one has
\[
g_{\gamma_\Theta(r)}=e^{\Psi_1(\Theta)}\,\widetilde g_{\gamma_\Theta(r)}\qquad (r>0).
\]
By continuity of $G$ and $\widetilde g$ at $x$,
\[
G_x=\lim_{r\downarrow 0} g_{\gamma_\Theta(r)}
     =e^{\Psi_1(\Theta)}\,\lim_{r\downarrow 0}\widetilde g_{\gamma_\Theta(r)}
     =e^{\Psi_1(\Theta)}\,\widetilde g_x .
\]
The left-hand side does not depend on $\Theta$, while $\widetilde g_x$ is fixed; hence $\Psi_1(\Theta)$ is the same for all $\Theta\in U_x\widetilde M$. Therefore $\Psi_1$ is constant on $U_x\widetilde M$.

\smallskip

\end{proof}

\subsection{Proof of Theorem \ref{thm:LAC_SY}}

\begin{proof}
\emph{Geodesic polar coordinates and notation.}
Fix $x\in\widetilde M$ and write
\[
\Phi:(0,R)\times U_x\widetilde M\to B(x;R)\setminus \{x\},\qquad \Phi(r,\Theta)=\exp_x^{\tilde g}(r\Theta).
\]
Set $y=\Phi(r,\Theta)$, $\gamma_\Theta(r)=\Phi(r,\Theta)$, and $\partial_r:=\dot\gamma_\Theta(r)$. 
By the Gauss lemma, $\|\partial_r\|_{\widetilde g}=1$ and $\partial_r\perp T_yS_r$, where $S_r:=\{z:\widetilde d(x,z)=r\}$ 
\citet{Petersen2016}(see Lemma~5.5.5). In these coordinates,
\begin{equation}\label{eq:polar-split}
\Phi^*\widetilde g = dr^2 + G_{ij}(r,\Theta)\,d\Theta^i d\Theta^j,
\qquad 
G_{ij}(r,\Theta)=\widetilde g\bigl(\Phi_*\partial_{\Theta^i},\Phi_*\partial_{\Theta^j}\bigr).
\end{equation}

\emph{Leading asymptotics on geodesic spheres.}
Pick a coordinate chart $\Theta:=(\Theta_1,\dots,\Theta_{d-1})$ on the unit sphere $S:=U_x\widetilde M$. Let $\partial_{\Theta_i}$, $1\le i\le d-1$, be the corresponding coordinate vector fields on $S$. Let \[ Y_i(r,\Theta):=\Phi_*(\partial_{\Theta_i}) \] be the corresponding Jacobi fields along $\gamma_\Theta$. Then $Y_i(0)=0$ and $Y_i'(0)=E_i$, where \[ E_i:=(\partial_{\Theta_i})_\Theta\in T_\Theta S\subset T_x\widetilde M, \] and hence $E_i\perp\Theta$.
Using the Jacobi equation and parallel transport $P_{0\to r}$, one has the Taylor expansions
\begin{equation}\label{eq:Jacobi}
P_{r\to 0}\,Y_i(r)= r\,E_i - \tfrac{r^3}{6}\,R_x(E_i,\Theta)\Theta + O(r^4),
\quad
P_{r\to 0}\,Y_i'(r)= E_i - \tfrac{r^2}{2}\,R_x(E_i,\Theta)\Theta + O(r^3),
\end{equation}
hence $Y_i'(r)-\frac1rY_i(r)=O(r^2)$ (after transporting to the same fiber). 
The normal–coordinate expansion $\widetilde g_{ab}(y)=\delta_{ab}-\frac13 R_{acbd}(x)y^cy^d+O(|y|^3)$
\citet{Petersen2016}[Lemma~5.5.7; Ex.~5.9.42 (5)--(6)] together with \eqref{eq:Jacobi} yields
\begin{equation}\label{eq:Gij}
G_{ij}(r,\Theta)=r^2 h_{ij}(\Theta)+O(r^4),\qquad 
G^{ij}(r,\Theta)=\frac1{r^2}h^{ij}(\Theta)+O(1),
\end{equation}
where $h_{ij}(\Theta)=\widetilde g_x(E_i,E_j)$ is the metric on $U_x\widetilde M$ induced by $\widetilde g_x$ (and $h^{ij}$ its inverse). 
Formulas \eqref{eq:polar-split}–\eqref{eq:Gij} are uniform in $\Theta$ on compact subsets.

\emph{Gradient and Hessian of an angular function.}
For ease of computation for curvature of conformal metrics, call $\Psi_1=2u:U_x\widetilde M \to \cl$. Denote by $S$ the unit sphere $U_x\widetilde{M}$ in $T_x\widetilde{M}.$ From \eqref{eq:polar-split}–\eqref{eq:Gij},
\begin{equation}\label{eq:grad2}
\|\nabla_{\widetilde g} u\|_{\widetilde g}^2
=(\partial_r u)^2+G^{ij}\,\partial_{\Theta^i}u\,\partial_{\Theta^j}u
=\frac1{r^2}\,\|\nabla_S u(\Theta)\|_{h}^2 + O(1),
\end{equation}
and $du(\partial_r)=0$, $du(E)=\frac1{r}W(u)+O(1)$ when $E\in T_yS_r$ is unit and corresponds to the $h$–unit $W\in T_\Theta U_x\widetilde M$ (via $Y'(0)=W$ and \eqref{eq:Jacobi}). 

For the Hessian we use $\mathrm{Hess}_{\widetilde g}u(X,Y)=X(Yu)-(\nabla^{\widetilde g}_X Y)u$. 
Using the standard polar-coordinate identities in geodesic normal coordinates, namely that radial curves are geodesics and that the geodesic spheres have induced metric \(\widetilde g|_{S_r}=r^2h+O(r^4)\), we obtain, as $r\downarrow0$,
\begin{equation}\label{eq:Hess-scalings}
\mathrm{Hess}_{\widetilde g}u(\partial_r,\partial_r)=0,\qquad
\mathrm{Hess}_{\widetilde g}u(\partial_r,E)= -\frac{1}{r^2}W(u)+O\!\Big(\frac1r\Big),\qquad
\mathrm{Hess}_{\widetilde g}u(E,E)=\frac{1}{r^2}\,\mathrm{Hess}_S u(\Theta)[W,W]+O\!\Big(\frac1r\Big).
\end{equation}
To see the last identity: write $E=Y/\|Y\|$ with $Y$ as above. 
Then $E=\frac1r Y+O(r)$, $\nabla_{\partial_r}E=O(r)$ (from $Y'-\frac1rY=O(r^2)$), and the Levi–Civita connection on $(S_r,\widetilde g|_{S_r})$ scales as $\nabla^{S_r}_E E = \frac1r\,\nabla^{S}_W W + O(1)$ because $\widetilde g|_{S_r}=r^2 h + O(r^4)$; hence
\[
E(Eu)-(\nabla^{\widetilde g}_E E)u
=\frac1{r^2}\bigl[\,W(Wu)-(\nabla^S_W W)u\,\bigr]+O\!\Big(\frac1r\Big)
=\frac1{r^2}\,\mathrm{Hess}_S u(\Theta)[W,W]+O\!\Big(\frac1r\Big).
\]

\emph{Sectional curvature under a conformal change.}
For any $\widetilde g$–orthonormal pair $e_1,e_2$ spanning a two–plane $\Pi\subset T_y\widetilde M$,
\begin{equation}\label{eq:Kconf}
K_{g}(\Pi)=e^{-2u(\Theta)}\Big(
K_{\widetilde g}(\Pi)
-\mathrm{Hess}_{\widetilde g}u(e_1,e_1)-\mathrm{Hess}_{\widetilde g}u(e_2,e_2)
+(du(e_1))^2+(du(e_2))^2-\|\nabla_{\widetilde g}u\|^2
\Big).
\end{equation}
This plane–wise identity is the standard specialization of the curvature tensor transformation under $g=e^{2u}\widetilde g$; see \citet{lee_introduction_to_RM_2018}, P.217, Theorem 7.30, or \citet{Besse1987}, P.58, Theorem 1.159.

\smallskip
\emph{Proof of (i).}
If $u$ is constant, then $\nabla_{\widetilde g}u\equiv0$ and $\mathrm{Hess}_{\widetilde g}u\equiv0$, so \eqref{eq:Kconf} gives $K_g=e^{-2u}K_{\widetilde g}$. 
Since $\widetilde g$ is smooth near $x$, $|K_{\widetilde g}|$ is bounded on $B_{\widetilde g}(x;\varepsilon)$; hence $\kappa_{\varepsilon}(s)=O(1)$ and $\int_0^\varepsilon s\,\kappa_{\varepsilon}(s)\,ds<\infty$.

\smallskip
\emph{Proof of (ii): lower bound.}
Because $u$ is not constant on the compact sphere $S:=U_x\widetilde M$, the Hessian $\mathrm{Hess}_S u$ is not identically zero (since Hessian zero implies harmonic, and on compact manifolds, harmonic implies constant); choose $\Theta_0\in S$ and a unit $W\in T_{\Theta_0}U_x\widetilde M$ with 
\[
\eta:=|\mathrm{Hess}_S u(\Theta_0)[W,W]|>0.
\]
Fix $y=\exp_x^{\tilde g}(r\Theta_0)$ with $r>0$ small, and let $E\in T_yS_r$ be the $\widetilde g$–unit vector corresponding to $W$ (via parallel transport along the radial geodesic with direction $\Theta_0$) as above. 
Choose the plane $\Pi=\mathrm{span}\{\partial_r,E\}$ and apply \eqref{eq:Kconf} with $e_1=\partial_r$, $e_2=E$. 
Since we have shown (see right after Equation \ref{eq:grad2}): $du(\partial_r)=0$, $du(E)=\frac1r W(u)+O(1)$, and \eqref{eq:grad2},
\[
\begin{aligned}
K_g(\Pi)&=e^{-2u(\Theta_0)}\Big(
K_{\widetilde g}(\Pi)
-\mathrm{Hess}_{\widetilde g}u(E,E) + (du(E))^2 - \|\nabla_{\widetilde g}u\|^2\Big)+O(1)\\
&=e^{-2u(\Theta_0)}\Big(
K_{\widetilde g}(\Pi)
-\frac{1}{r^2}\,\mathrm{Hess}_S u(\Theta_0)[W,W]
+\frac{1}{r^2}\,(W(u))^2
-\frac{1}{r^2}\,\|\nabla_S u(\Theta_0)\|_h^2
\Big)+O(1).
\end{aligned}
\]
The last two terms combine to $-\frac{1}{r^2}\|( \nabla_S u(\Theta_0))_{\perp W}\|_h^2\le0$. 
Since $K_{\widetilde g}$ is bounded and $\eta>0$, there exist constants $c_1,C_1>0$ and $r_0>0$ such that for $0<r<r_0$,
\[
|K_g(\Pi)|\ \ge\ \frac{c_1}{r^2}-C_1.
\]
Hence $\kappa_{\varepsilon}(r)\ge c_1 r^{-2}-C_1$ for $0<r<r_0$.

\smallskip
\emph{Proof of (ii): upper bound.}
Let $y$ satisfy $\widetilde d(x,y)=r$ and let $\Pi\subset T_y\widetilde M$ be any two–plane with a $\widetilde g$–orthonormal basis $(e_1,e_2)$. 
Using \eqref{eq:Kconf}, the boundedness of $K_{\widetilde g}$, and the estimates \eqref{eq:grad2}–\eqref{eq:Hess-scalings},
\[
|K_g(y,\Pi)|
\ \le\ C_0 
+ e^{-2\inf u}\Big(
\frac{C_2}{r^2} + \frac{C_3}{r^2} + \frac{C_4}{r^2}
\Big)
\ \le\ \frac{C}{r^2}+C,
\]
where the constants depend only on local bounds of $u$, $\nabla_S u$, $\mathrm{Hess}_S u$ on $U_x\widetilde M$, and on $K_{\widetilde g}$ near $x$. 
Therefore $\kappa_{\varepsilon}(r)\le C r^{-2}+C$ for small $r$. After decreasing $r_0$ if necessary and renaming constants, the additive constants in the lower and upper bounds are absorbed into the $r^{-2}$ terms; hence
\[
c r^{-2}\le \kappa_{\varepsilon}(r)\le C r^{-2}\qquad (0<r<r_0),
\]
which is the two-sided estimate claimed in (ii).
Finally, $\int_0^\varepsilon s\,\kappa_{\varepsilon}(s)\,ds=\infty$ follows from comparison with $\int_0^{r_0} s^{-1}\,ds$.

\end{proof}

\subsection{Proof of Theorem \ref{thm:blowup of GL near isolated singularity}}

We begin with the proposition below whose proof we leave to the reader:

\begin{proposition}\label{prop:truncation-error}
Let $(M,d,\mu)$ be a metric measure space and $\eta \in (0,1/2)$. Then for any $t>0$,
\[
\left| \frac{1}{t^{d/2+1}} \int_{M\backslash B_{t^\eta }(x)} \exp\left(-\frac{d^2(x,y)}{t}\right) (f(x)-f(y)) p(y)d \mathrm{\mu}(y) \right| 
\leq [|f(x)|\|p\|_{1} + \norm{fp}_{1}] \, \frac{1}{t^{d/2 +1}} e^{-t^{2\eta -1}}
\]
for any $p \in L^1(\mu), f \in \cC(M)\cap L^1(M,p \mu)$ and $x \in X$. As a consequence, we get that
\[
L_tf(x) = \frac{1}{t^{d/2+1}} \int_{B_{t^\eta }(x)} \exp\left(-\frac{d^2(x,y)}{t}\right) (f(x)-f(y)) p(y)d \mathrm{\mu}(y) + O(t^{-d/2-1}e^{-t^{2 \eta -1}}), t \downarrow 0.
\] 
\end{proposition}

\begin{proof}
For $y \notin B_{t^\eta}(x)$, one has $d(x,y)^2/t \ge t^{2\eta-1}$. Hence
\(
e^{-d(x,y)^2/t} \le e^{-t^{2\eta-1}} \quad \text{on } M \setminus B_{t^\eta}(x).
\)
Therefore
\(
\left| \int_{M \setminus B_{t^\eta}(x)} e^{-d(x,y)^2/t}\,(f(x)-f(y))\,p(y)\,d\mu(y) \right|
\le e^{-t^{2\eta-1}} \bigl( |f(x)|\|p\|_1 + \|fp\|_1 \bigr).
\)
Multiplying by $t^{-d/2-1}$ gives the bound. The consequence part is trivial.
\end{proof}

Armed with the above, we prove Theorem \ref{thm:blowup of GL near isolated singularity} below.

\begin{proof}
Fix $\delta > 0, \eta\in(\frac{1}{2+\delta},\frac{1}{2})$ and set $R_t:=t^{\eta}$. Using Proposition \ref{prop:truncation-error}, split
\[
L_t^g f(x)
=\frac{1}{t^{d/2+1}}\!\int_{M} e^{-d_g(x,y)^2/t}\bigl(f(x)-f(y)\bigr)p(y)\,d\mathrm{vol}_g(y)
=I_t^{(1)}+I_t^{(2)},
\]
where $I_t^{(1)}$ is the integral over the $\tilde g$-metric ball $B_{\widetilde g}(x,R_t)$ and $I_t^{(2)}=O(t^{-d/2-1}e^{-t^{2 \eta -1}})$ as $t \downarrow 0$.

\medskip\noindent
\textit{Local expansion.}  
Write $y=\exp_x^{\widetilde g}(r,\Theta)$ with $r\in(0,\rho)$ and $\Theta\in U_x\widetilde M$.  
By assumption,
\begin{equation}\label{eqn:local-expansion}
d_g(x,y)^2 = L(\Theta)^2 r^2 + E(r,\Theta), \qquad |E(r,\Theta)|\le Cr^{2+\delta}, \delta > 0.
\end{equation}
The volume form satisfies
\begin{equation}\label{eqn:loc-expansion-vol-form}
d\mathrm{vol}_g(y)
= e^{\frac d2\Psi_1(\Theta)}\,r^{d-1}\bigl(1+O(r^2)\bigr)\,dr\,d\Theta,
\qquad r\to0.                                         
\end{equation}

\nd Taylor expansions of $f$ and $p$ give
\begin{align}
f(\exp_x^{\widetilde g}(r,\Theta))
&= f(x) + r\,\langle \nabla f(x),\Theta\rangle_{\widetilde g}
 + \tfrac12 r^2 \bigl(\nabla^2 f(x)\bigr)(\Theta,\Theta) + O(r^3),\\
p(\exp_x^{\widetilde g}(r,\Theta))
&= p(x) + r\,\langle \nabla p(x),\Theta\rangle_{\widetilde g} + O(r^2).          
\end{align}

\nd Thus
\begin{equation}
\bigl(f(x)-f(\exp_x^{\widetilde g}(r,\Theta))\bigr)\,
p(\exp_x^{\widetilde g}(r,\Theta))
= - r\,p(x)\langle \nabla f(x),\Theta\rangle_{\widetilde g}
  - r^2\,\Phi(\Theta)
  + O(r^3),
\end{equation}
where $\Phi(\Theta)$ is as in the Theorem \ref{thm:blowup of GL near isolated singularity}.

\nd From Equation \ref{eqn:local-expansion},
\[
e^{-d_g(x,y)^2/t}
= e^{-L(\Theta)^2 r^2/t}\bigl(1+O(E(r,\Theta)/t)\bigr).
\]
Since $E(r,\Theta)=O(r^{2+\delta})$ and $r\le R_t=t^\eta$, we have
\[
E(r,\Theta)/t = O(t^{(2+\delta)\eta-1}).
\] 

\nd We decompose
\[
I_t^{(1)} = A_t+B_t+C_t,
\]
where $A_t$ is the first-order term in $r$, $B_t$ is the second-order term, and $C_t$ is the remainder.

\medskip\noindent
\textit{Leading term $A_t$.}
\[
A_t
= -\frac{p(x)}{t^{d/2+1}}
\int_{U_x\widetilde M}\! e^{\frac d2\Psi_1(\Theta)}
\langle \nabla f(x),\Theta\rangle_{\widetilde g}
\int_0^{R_t} r^{d} e^{-L(\Theta)^2 r^2/t}\,dr\,d\sigma(\Theta).
\]
Let $u=L(\Theta)r/\sqrt t$. Then
\[
\int_0^{R_t} r^{d}e^{-L(\Theta)^2 r^2/t}\,dr
= t^{(d+1)/2} L(\Theta)^{-(d+1)}
  \int_0^{L(\Theta)R_t/\sqrt t} u^{d} e^{-u^2}\,du.
\]
As $L(\Theta)R_t/\sqrt t\to\infty$, the above substitution by $u$ and properties of incomplete Gamma function: $\int_a^\infty u^m e^{-u^2}\,du \le C_m a^{m-1}e^{-a^2}, \, a\ge1,$ yield:

\[
\int_0^{R_t} r^{d}e^{-L(\Theta)^2 r^2/t}\,dr
= c_d\,L(\Theta)^{-(d+1)} t^{(d+1)/2} + O\!\left(t^{(d+1)/2}e^{-ct^{2\eta-1}}\right), c>0\text{ a constant independent of }t
\]
with $c_d=\tfrac12\Gamma(\frac{d+1}{2})$.
Thus
\begin{equation}
A_t
= -c_d\,p(x)\,\langle \nabla f(x), b(x)\rangle_{\widetilde g}\,t^{-1/2}
  + O\!\bigl(t^{-1/2}e^{-ct^{2\eta-1}}\bigr).
\end{equation}
\medskip\noindent
\textit{Second--order term $B_t$.}
Similarly,
\[
B_t
= -\frac{1}{t^{d/2+1}}
\int_{U_x\widetilde M} e^{\frac d2\Psi_1(\Theta)}\Phi(\Theta)
\int_0^{R_t} r^{d+1} e^{-L(\Theta)^2 r^2/t}\,dr\,d\sigma(\Theta).
\]
Substituting $u=L(\Theta)r/\sqrt t$ gives
\[
\int_0^{R_t} r^{d+1} e^{-L(\Theta)^2 r^2/t}\,dr
= t^{\frac{d+2}{2}} L(\Theta)^{-(d+2)}
   \int_0^{L(\Theta)R_t/\sqrt t} u^{d+1} e^{-u^2}\,du.
\]
Since $R_t=t^\eta$, the upper limit is
\[
a_\Theta(t):=\frac{L(\Theta)R_t}{\sqrt t}
= L(\Theta)t^{\eta-\frac12}.
\]
Because $L$ is positive and continuous on the compact set $U_x\widetilde M$, there exists
$L_{\min}>0$ such that $L(\Theta)\ge L_{\min}$ for all $\Theta\in U_x\widetilde M$. Hence
$a_\Theta(t)\ge L_{\min}t^{\eta-\frac12}\to\infty$ uniformly in $\Theta$ as $t\downarrow0$. Moreover,
\[
a_\Theta(t)^2=L(\Theta)^2t^{2\eta-1}\ge L_{\min}^2t^{2\eta-1}.
\]
Using the standard tail estimate
\[
\int_a^\infty u^{d+1}e^{-u^2}\,du
\le C_d\,a^d e^{-a^2}\qquad (a\ge1),
\]
and absorbing the polynomial factor into the exponential by decreasing the constant in the exponent,
we obtain, for some $c>0$ independent of $t$ and $\Theta$,
\[
\int_{a_\Theta(t)}^\infty u^{d+1}e^{-u^2}\,du
=O\!\left(e^{-c t^{2\eta-1}}\right).
\]
Therefore, uniformly in $\Theta$,
\[
\int_0^{L(\Theta)R_t/\sqrt t} u^{d+1} e^{-u^2}\,du
= \int_0^\infty u^{d+1} e^{-u^2}\,du
  + O\!\left(e^{-c t^{2\eta-1}}\right).
\]
Thus
\[
\int_0^{R_t} r^{d+1} e^{-L(\Theta)^2 r^2/t}\,dr
= c_{d+1}\,L(\Theta)^{-(d+2)} t^{(d+2)/2}
  + O\!\left(t^{(d+2)/2} e^{-c t^{2\eta-1}}\right),
\]
where
\[
c_{d+1}:=\int_0^\infty u^{d+1}e^{-u^2}\,du
=\frac12\Gamma\!\left(\frac{d+2}{2}\right).
\]


\medskip\noindent
\textit{Remainder $C_t$.}
Remaining terms involve $\int_0^{R_t} r^{d+2} e^{-c r^2/t}dr$, which satisfies
\[
\frac{1}{t^{d/2+1}}\int_0^{R_t} r^{d+2} e^{-c r^2/t}dr= O(\sqrt t),
\]
hence $C_t=O(\sqrt t)$. To see this, substitute $r=\sqrt{t}\,u$,
\begin{align*}
C_t:=\frac{1}{t^{d/2+1}}\int_0^{R_t} r^{d+2} e^{-cr^2/t}\,dr
&= t^{1/2}\!\int_0^{R_t/\sqrt t} u^{d+2} e^{-cu^2}\,du \\[-0.1cm]
&\le t^{1/2}\!\int_0^\infty u^{d+2} e^{-cu^2}\,du
= O(\sqrt t),
\end{align*}
since $R_t/\sqrt t=t^{\eta-1/2}\to\infty$ for $\eta\in(1/4,1/2)$.

\medskip\noindent
\textit{Effect of the error term $E(r,\Theta)$.}
Finally, since $E(r,\Theta)=O(r^{2+\delta})$, we have, uniformly for $0\le r\le R_t$,
\[
\bigl|e^{-E(r,\Theta)/t}-1\bigr|
\le C\,\frac{r^{2+\delta}}{t}.
\]
The contribution of this factor to the leading term $A_t$ is bounded by
\[
\frac{C}{t^{d/2+1}}\int_0^{R_t} r^{d}e^{-cr^2/t}\frac{r^{2+\delta}}{t}\,dr
=O(t^{-1/2+\delta/2}).
\]
The corresponding contribution to the second-order term $B_t$ is
\[
\frac{C}{t^{d/2+1}}\int_0^{R_t} r^{d+1}e^{-cr^2/t}\frac{r^{2+\delta}}{t}\,dr
=O(t^{\delta/2}),
\]
and is therefore of higher order. Thus this error contribution is $o(1)$ when $\delta>1$, and is $O(\sqrt t)$ when $\delta\ge2$.
\medskip
\nd \nd Finally, combining the estimates for $A_t$, $B_t$, $C_t$, the error contribution from $E(r,\Theta)$, and $I_t^{(2)}$, we obtain, in general,
\[
\sqrt{t}\,L_t^{g,\mathrm{int}}f(x)=O(1), \qquad t\downarrow0.
\]
Furthermore, for $\delta\ge2$, we have:
\[
L_t^{g,\mathrm{int}} f(x)
= -c_d\,p(x)\,\langle\nabla_{\widetilde g} f(x),b(x)\rangle_{\widetilde g}\, t^{-1/2}
  - c_{d+1}B_0(x)
  + O(\sqrt t),
\]
as claimed. 
\end{proof}

\subsection{Taylor expansion of extrinsic graph Laplacian }

In this section we prove Theorem \ref{thm:TE extrinsic GL}.
\begin{proof}

\nd \textbf{Step 1: infinitesimal metric and volume comparison:}

\nd For any $r>0,$ denote by $B_M^D(x;r)$ the intersection of the Euclidean ball with $M$, i.e. $B_{\Ddim}(x;r)\cap M,$
and $\frac{M-x}{\sqrt{t}}$ the set $\{\frac{m-x}{\sqrt{t}}:m\in M\}.$
Consider the orthogonal projection $\pi_x$ from $\setR^D$ onto the translated tangent space $x + T_xM$. Applying a rigid motion if needed, we can assume that $x=0_D$ and $x+T_xM=\Ddim \times \{0_{D-d}\} \simeq \setR^d$, so that $\pi_x$ can be seen as the orthogonal projection mapping a vector of $\Ddim$ onto its first $d$ coordinates. We let $\epsilon>0$ be such that $\pi_x$ is a smooth diffeomorphism of $B_{\widetilde{M}}(x;\sqrt{\epsilon}) \cap M$ onto its image
 $\Omega \df \pi_x (B_{\widetilde{M}}(x;\sqrt{\epsilon}) \cap M) \subset \ddim.$
 Following \cite{CoifmanLafon2006}, for a generic $y \in B_{\widetilde{M}}(x;\sqrt{\epsilon}) \cap M$ we set $u=(u_1 \dots u_d):=\pi_x(y).$ Note that $\pi_x$ acts as a local chart centered at $x,$ so that $\operatorname{int}(\Omega)$ is an open subset of $\ddim$ with $0_d=\pi_x(x)$ as interior or $\cC^0$ boundary point. Below, we express the Euclidean distance and the Riemannian volume measure on $M$ in the $u$-coordinates introduced above. We let $Q_{x,m}$ denote a generic homogeneous polynomial of degree $m.$

\begin{lemma}\label{lem:prep}
Fix $0< \eta < 1/2.$ Then as $t \downarrow 0$, for any $u \in \pi(\mathbb{B}_{t^\eta}^D(x) \cap M)$,
 \begin{align*}
    \norm{\pi^{-1}(u)}_{\setR^D}^2 &= \norm{u}_{\ddim}^2 + Q_{x,4}(u) + Q_{x,5}(u) + O(t^{6\eta}) && \text{(metric comparison)} \\
    \rho(u) 
    &= 1 + Q_{x,2}(u) + Q_{x,3}(u) + O(t^{4\eta}) && \text{(infinitesimal volume comparison)}
\end{align*}
where $\rho \in L^1(\Omega,\mathcal{L}^d)$ is the density of the push-forward measure $\pi_\# \mathrm{vol}_g$ on $\Omega$ with respect to $\mathcal{L}^d$.
\end{lemma}
 
\begin{proof}
    It follows line by line from \cite{CoifmanLafon2006} (see Appendix B, Lemma 7), with the assumption $\norm{y-x}<t^{\eta}$ instead of $\norm{y-x}<t^{1/2}$. This is weaker for small $t>0$, since $t^{\eta}>t^{1/2}$ when $0<\eta<1/2$.
\end{proof}

\nd \textbf{Step 2:}

\begin{lemma}[\textbf{comparing angles after projection and after inverse exponential map}]\label{lem:comparing projection and log}
Let $\widetilde{M} \subset \mathbb{R}^D$ be a smooth embedded submanifold endowed with the Riemannian metric induced from $\mathbb{R}^D$. 
Fix $x \in \widetilde{M}$ and let $\pi_x : \mathbb{R}^D \to T_x\widetilde{M}$ denote the orthogonal projection onto the tangent space at $x$, translated by $x,$ i.e. onto the affine subspace $x+T_x\widetilde{M}$ of $\Ddim.$
For any point $y$ in a sufficiently small normal neighborhood $B(x;r)$ of $x$, i.e. a neighborhood on which $\exp_x^{\tilde g}$ is a diffeomorphism, define $v:=\log_x^{\tilde g}y\in T_x\widetilde M$ and $s=d_{\tilde g}(x,y)=\|v\|_{\tilde g}$.
Then for all $y \in \widetilde{M}$ with $s < r$, we have:
\[
\frac{\pi_x(y)}{\|\pi_x(y)\|}
=\frac{\log_x^{\tilde g}y}{\|\log_x^{\tilde g}y\|_{\tilde g}}
+O(s^{2}),
\qquad \text{as } y \to x.
\]

\end{lemma}

\begin{proof}
    
We work in normal coordinates at $x$ on $\widetilde{M}$. For $y$ in a normal neighborhood, write
$v=\log_x^{\tilde g}y\in T_x\widetilde{M}$ and $s=\|v\|=d_{\tilde g}(x,y)$. 
The standard extrinsic expansion of the embedded exponential map yields
\[
y-x=v+\frac12\,\mathrm{II}_x(v,v)+R_3(v),
\]
where $\mathrm{II}_x(v,v)\in (T_x\widetilde M)^\perp$ is the second fundamental form at $x$, and $\|R_3(v)\|\le Cs^3$ for $s$ sufficiently small. This follows by Taylor expanding the fixed isometric embedding $\Phi:\widetilde M\to\mathbb R^D$ along the geodesic $\gamma$ with $\gamma(0)=x$ and $\dot\gamma(0)=v$, and using the Gauss--Weingarten formula; see, for instance, \citet{CoifmanLafon2006}, Appendix B, Lemma 7.


Applying the orthogonal projection $\pi_x:\mathbb{R}^D \to T_x\widetilde{M}$ to both sides gives
\[
\pi_x(y - x)
\;=\;
\pi_x(v)
\;+\;
\tfrac{1}{2}\,\pi_x\!\big(\mathrm{II}_x(v,v)\big)
\;+\;
\pi_x R_3(v).
\]
Now use the facts that 
$\pi_x(v) = v$ (since $v\in T_x\widetilde{M}$) and 
$\pi_x(\mathrm{II}_x(v,v)) = 0$ (since $\mathrm{II}_x(v,v)$ is normal). 
Setting $r := \pi_x R_3(v)$ yields
\[
\pi_x(y - x)
\;=\;
v + r,
\qquad
\|r\| \le C s^3\quad\Longrightarrow\quad
\big|\|\pi_x(y)\|-s\big| \le C s^{3},
\qquad\text{and thus}\qquad
\|\pi_x(y)\|\ge s-C s^{3}\ge \tfrac12 s, s \text{ small.}
\]
Let $w:=\pi_x(y)=v+r$. Then
\[
\left\|\frac{w}{\|w\|}-\frac{v}{\|v\|}\right\|
\;\le\;
\frac{\|r\|}{\|v\|} 
\;+\;
\|v\|\left|\frac{1}{\|w\|}-\frac{1}{\|v\|}\right|
\;\le\;
\frac{C s^{3}}{s}
\;+\;
s \cdot \frac{|\|w\|-\|v\||}{\|w\|\,\|v\|}
\;\le\;
C s^{2} + s\cdot\frac{C s^{3}}{(\tfrac12 s)\,s}
\;=\;O(s^{2}),
\]
where we used $\|r\|\le C s^3$, $|\|w\|-\|v\||\le C s^3$, and $\|w\|\ge s/2$ for $s$ small.
This is equivalent to
\[
\left\|\frac{w}{\|w\|}-\frac{v}{\|v\|}\right\|
= O(s^{2}),
\qquad s=d_{\tilde g}(x,y)\to 0,
\]
as claimed.
\end{proof}
\nd \textbf{Step 3: }Fix $\epsilon>0, 1/4 < \eta < 1/2,$ and consider $t>0$ small enough so $t^{\eta}< \sqrt{\epsilon}$. Then
\begin{equation}
\begin{aligned}
& L^{g, ext}_tf(x)\\
 &:=\frac{1}{t^{d/2+1}}
\int_{M} e^{-\,\frac{\|x-y\|_{\Ddim}^{2}}{t}} \big(f(x)-f(y)\big)\,p(y)\,d\mathrm{vol}_{g}(y)\\
&= \frac{1}{t^{d/2+1}}
\int_{B_M^D(x;t^{\eta})} e^{-\,\frac{\|x-y\|_{\Ddim}^{2}}{t}} \big(f(x)-f(y)\big)\,p(y)\,d\mathrm{vol}_{g}(y) +O(t^{-d/2-1}e^{-t^{2\eta-1}}) [\text{cf. Proposition}\,\ref{prop:truncation-error}]  \\
&=\frac{1}{t^{d/2+1}}
\int_{B_M^D(x;t^{\eta})} e^{-\,\frac{\|x-y\|_{\Ddim}^{2}}{t}} \big(f(x)-f(y)\big)\,p(y)\, e^{\frac{d}{2}\Psi_1(\Theta(y))}d\mathrm{vol}_{\tilde{g}}(y) + O(t^{-d/2-1}e^{-t^{2\eta-1}}), \Theta(y):=\frac{\log_x^{\tilde g} y}{\|\log_x^{\tilde g} y\|_{\tilde g}}\\
&= \frac{1}{t^{d/2+1}}
\int_{B_M^D(x;t^{\eta})} e^{-\,\frac{\|x-y\|_{\Ddim}^{2}}{t}} \big(f(x)-f(y)\big)\,p(y)\, e^{\frac{d}{2}\Psi_1\left(\frac{\pi_x(y)}{\norm{\pi_x(y)}} + O(d^2_{\tilde{g}}(x,y))\right)}d\mathrm{vol}_{\tilde{g}}(y)+O(t^{-d/2-1}e^{-t^{2\eta-1}}) , (\text{cf. Lemma \ref{lem:comparing projection and log}})\\
 &= \frac{1}{t^{d/2+1}}
\int_{B_M^D(x;t^{\eta})} e^{-\,\frac{\|x-y\|_{\Ddim}^{2}}{t}} \big(f(x)-f(y)\big)\,p(y)\, e^{\frac{d}{2}\Psi_1\left(\frac{\pi_x(y)}{\norm{\pi_x(y)}} + O(\norm{x-y}^2)\right)}d\mathrm{vol}_{\tilde{g}}(y)+ O(t^{-d/2-1}e^{-t^{2\eta-1}}) \\
&=  \frac{1}{t^{d/2+1}}
\int_{B_M^D(x;t^{\eta})} e^{-\,\frac{\|x-y\|_{\Ddim}^{2}}{t}} \big(f(x)-f(y)\big)\,p(y)\, e^{\frac{d}{2}\Psi_1\left(\frac{\pi_x(y)}{\norm{\pi_x(y)}} + O(t^{2\eta})\right)}d\mathrm{vol}_{\tilde{g}}(y)+ O(t^{-d/2-1}e^{-t^{2\eta-1}}), (\because \norm{x-y}<t^{\eta} )\\  
&\text{(Using compactness and equivalence of the intrinsic and extrinsic distances on the compact ball)}\\
&=\frac{1}{t^{d/2+1}}
\int_{\Omega \cap B^d(0;t^{\eta})}
\exp\left(
    -\frac{\|u\|_{\mathbb{R}^d}^2}{t}
    -\frac{ Q_{x,4}(u) + Q_{x,5}(u) + O(t^{6\eta}) }{t}
\right)
\left( \tilde{f}(0_d) - \tilde{f}(u) \right)
\tilde{p}(u)\,\\
&\exp\!\left(
    \frac{d}{2}\,\Psi_1\!\left(\frac{u}{\|u\|} + O(t^{2\eta}))\right)
\right)  \times  \left( 1 + Q_{x,2}(u) + Q_{x,3}(u) + O(t^{4\eta})\right) \, du +  O(t^{-d/2-1}e^{-t^{2\eta-1}}), \Omega:= \pi_x(B_M(x;\sqrt{\epsilon})), u:=\pi_x(y)\\
&= \frac{1}{t} \int_{\frac{\Omega}{\sqrt{t}} \cap B^d(0;t^{\eta-1/2})}
    e^{-\left( \|\zeta\|^2 
        + \left( \frac{Q_{x,4}(\sqrt{t}\,\zeta)
                   + Q_{x,5}(\sqrt{t}\,\zeta)
                   + O(t^{6\eta})}{t} \right) \right)}
    \left( \tilde{f}(0_d) - \tilde{f}(\sqrt{t}\,\zeta) \right)
    \tilde{p}(\sqrt{t}\,\zeta),\\
&\exp\!\left(
    \frac{d}{2}\,\Psi_1\!\left(\frac{\zeta}{\|\zeta\|} + O(t^{2\eta}))\right)
\right)  \times  \left( 1 + Q_{x,2}(\sqrt{t}\zeta) + Q_{x,3}(\sqrt{t}\zeta) + O(t^{4\eta})\right) \, d\zeta +  O(t^{-d/2-1}e^{-t^{2\eta-1}}),
    (\text{substituting }\zeta := u/\sqrt{t})
    \\
&= \frac{1}{t} \int_{\frac{\Omega}{\sqrt{t}} \cap B^d(0;t^{\eta-1/2})}
    e^{-\left( \|\zeta\|^2 
        + t Q_{x,4}(\zeta)
        + t^{3/2} Q_{x,5}(\zeta)
        + O(t^{6\eta-1}) \right)}
    \left( \tilde{f}(0_d) - \tilde{f}(\sqrt{t}\,\zeta) \right)
    \tilde{p}(\sqrt{t}\,\zeta)  \exp\!\left(
    \frac{d}{2}\,\Psi_1\!\left(\frac{\zeta}{\|\zeta\|} + O(t^{2\eta}))\right)
\right) \\
&\qquad\qquad\times 
    \left( 1 + t Q_{x,2}(\zeta) + t^{3/2} Q_{x,3}(\zeta) + O(t^{4\eta}) \right)
    \, d\zeta \qquad (\text{Using homogeneity of the polynomials $Q_m$})\\
&= \frac{1}{t}\int_{\frac{\Omega}{\sqrt{t}}\cap B(0,t^{\eta-1/2})}
   e^{-\|\zeta\|^2}\,
   e^{\frac{d}{2}\Psi_1(\zeta/\|\zeta\|)}\,
   \Bigl(-\sqrt{t}\,\langle\nabla_{\tilde{g}} f (x),\zeta\rangle
        -\tfrac{t}{2}\,\nabla^2_{\tilde{g}} f(x)(\zeta,\zeta)
        +O(t^{3/2}\|\zeta\|^3)\Bigr) \\
&\qquad\qquad\qquad\qquad \times
   \Bigl(p(x)+\sqrt{t}\,\langle\nabla_{\tilde{g}}  p(x),\zeta\rangle
        +O(t\|\zeta\|^2)\Bigr)\bigl(1+o(1)\bigr)\,d\zeta
\\
&= -\frac{1}{\sqrt{t}}\int_{\frac{\Omega}{\sqrt{t}}\cap B(0,t^{\eta-1/2})}
   e^{-\|\zeta\|^2}\,
   e^{\frac{d}{2}\Psi_1(\zeta/\|\zeta\|)}\,
   p(x)\,\langle\nabla_{\tilde{g}}  f(x),\zeta\rangle\,d\zeta
\\
&\quad
   -\int_{\frac{\Omega}{\sqrt{t}}\cap B(0,t^{\eta-1/2})}
   e^{-\|\zeta\|^2}\,
   e^{\frac{d}{2}\Psi_1(\zeta/\|\zeta\|)}\,
   \Bigl(\langle\nabla_{\tilde{g}} f(x),\zeta\rangle\langle\nabla_{\tilde{g}} p(x),\zeta\rangle
        +\tfrac{1}{2}p(x)\,\nabla_{\tilde{g}}^2 f(x)(\zeta,\zeta)\Bigr)\,d\zeta
   +O(t^{1/2})
\\
&= -\frac{1}{\sqrt{t}}\int_{U_x\widetilde M}
   e^{\frac{d}{2}\Psi_1(\Theta)}\,
   p(x)\,\langle\nabla_{\tilde{g}}  f(x),\Theta\rangle
   \left(\int_0^\infty e^{-r^2}r^{d}\,dr\right)
   d\sigma(\Theta)
\\
&\quad
   -\int_{U_x\widetilde M}
   e^{\frac{d}{2}\Psi_1(\Theta)}\,
   \Bigl(\langle\nabla_{\tilde{g}}  f(x),\Theta\rangle\langle\nabla_{\tilde{g}}  p(x),\Theta\rangle
        +\tfrac{1}{2}p(x)\,\nabla_{\tilde{g}} ^2 f(x)(\Theta,\Theta)\Bigr)
   \left(\int_0^\infty e^{-r^2}r^{d+1}\,dr\right)\bigl(1+o(1)\bigr)
   d\sigma(\Theta)
   +O(t^{1/2})
\\
&\text{(using polar decomposition } \zeta:=r\Theta,\ d\zeta=r^{d-1}dr\,d\sigma(\Theta);
\text{ the linear and quadratic terms in }\zeta \\
& \text{ give the radial powers }r^d
\text{ and }r^{d+1}\text{, respectively)}
\\
&= (1+o(1))\left(-\frac{c_d}{\sqrt{t}}\,p(x)\,B_M f(x)
   \;-\;c_{d+1}\Bigl(p(x)\,A_M f(x)
                    +r(p,f)_M(x)\Bigr)\right)
   +O(t^{1/2}) .
\end{aligned}
\end{equation}

\nd where in the last step we could pull the term $1+o(1), t\downarrow 0$ that is uniform w.r.t. $\zeta,$ outside of the integral if $\eta >1/4.$ Recall that above
\begin{align*}
c_d &:= \int_0^\infty e^{-r^2}r^{d}\,dr,
c_{d+1}:= \int_0^\infty e^{-r^2}r^{d+1}\,dr, B_M f(x)
:= \int_{U_x\widetilde M} e^{\frac{d}{2}\Psi_1(\Theta)}
    \,\langle\nabla_{\tilde{g}}  f(x),\Theta\rangle\,d\sigma(\Theta),\\
A_M f(x)
&:= \frac{1}{2}\int_{U_x\widetilde M} e^{\frac{d}{2}\Psi_1(\Theta)}
    \,\nabla^2_{\tilde{g}}  f(x)(\Theta,\Theta)\,d\sigma(\Theta),
r(p,f)_M(x):= \int_{U_x\widetilde M} e^{\frac{d}{2}\Psi_1(\Theta)}
    \,\langle\nabla_{\tilde{g}}  f(x),\Theta\rangle
     \,\langle\nabla_{\tilde{g}}  p(x),\Theta\rangle\,d\sigma(\Theta).
\end{align*}

\end{proof}

\section{Numerical simulations}\label{scn:simulations}

\subsection{Simulations of continuous \textit{intrinsic} graph Laplacian on a punctured disk}

The purpose of the following two simulations is to show that the graph Laplacian may \textit{not} detect or distinguish all types of \textit{nonremovable} singularities. Here we have two examples of such singularities, and the graph Laplacian only blows up in the first one as the bandwidth $t$ approaches $0,$ but not in the second one. Of course, Theorem \ref{thm:Asymptotics of Continuous graph Laplacian} shows us that graph Laplacian cannot detect removable singularities with controlled curvature blowups, cf. Definition \ref{def:controlled-blow-up of curvature}.

We illustrate the asymptotic limit of the  
$\sqrt{t}\,L_t f(x)$ on a two-dimensional locally angularly conformal (cf. Definition \ref{def:locally angularly conformal}) metric 
by direct numerical quadrature. We fix $p \equiv 1$ and consider 
\[
g = a(\theta)^2\,(dr^2 + r^2 d\theta^2), 
\qquad a(\theta) = 1 + 0.4\cos\theta, \qquad \Psi_1(\theta)=  2 \, \log \,a(\theta)
\]

\subsubsection{\textbf{Visualisation of the local metric geometry near the singular point.}}
To complement Table~\ref{tab:scaled_GL}, we include two figures that illustrate the local geometry of the angularly conformal metric used in the simulation. 
In the numerical experiment, the distance from the singular point to a nearby point $(r,\theta)$ is approximated by the length of the fixed-angle radial path, namely
\[
d_g\bigl(0,(r,\theta)\bigr)\approx a(\theta)\,r.
\]
Accordingly, for a fixed radius $\rho>0$, the corresponding approximate balls $B_g(0;\rho)$ centered at $(0,0)$ and of radius $\rho>0$ are given by 

$$\{(r, \theta): r< \frac{\rho}{a(\theta)}\}= \bigsqcup_{0\le \theta < 2\pi}\{r: r< \frac{\rho}{a(\theta)}\}.$$

\nd Thus, Figure \ref{fig:local-approx-balls-origin} displays approximate origin-centered metric balls $B_g(0;\rho)$  associated with the distance approximation used in the simulation. In particular, it makes visible the angular anisotropy of the metric: for fixed $\rho$, the boundary is no longer a Euclidean circle, but is stretched or compressed along the $\theta$-direction depending on the value of $a(\theta)$.

\begin{figure}[t]
\centering
\includegraphics[width=0.5\textwidth]{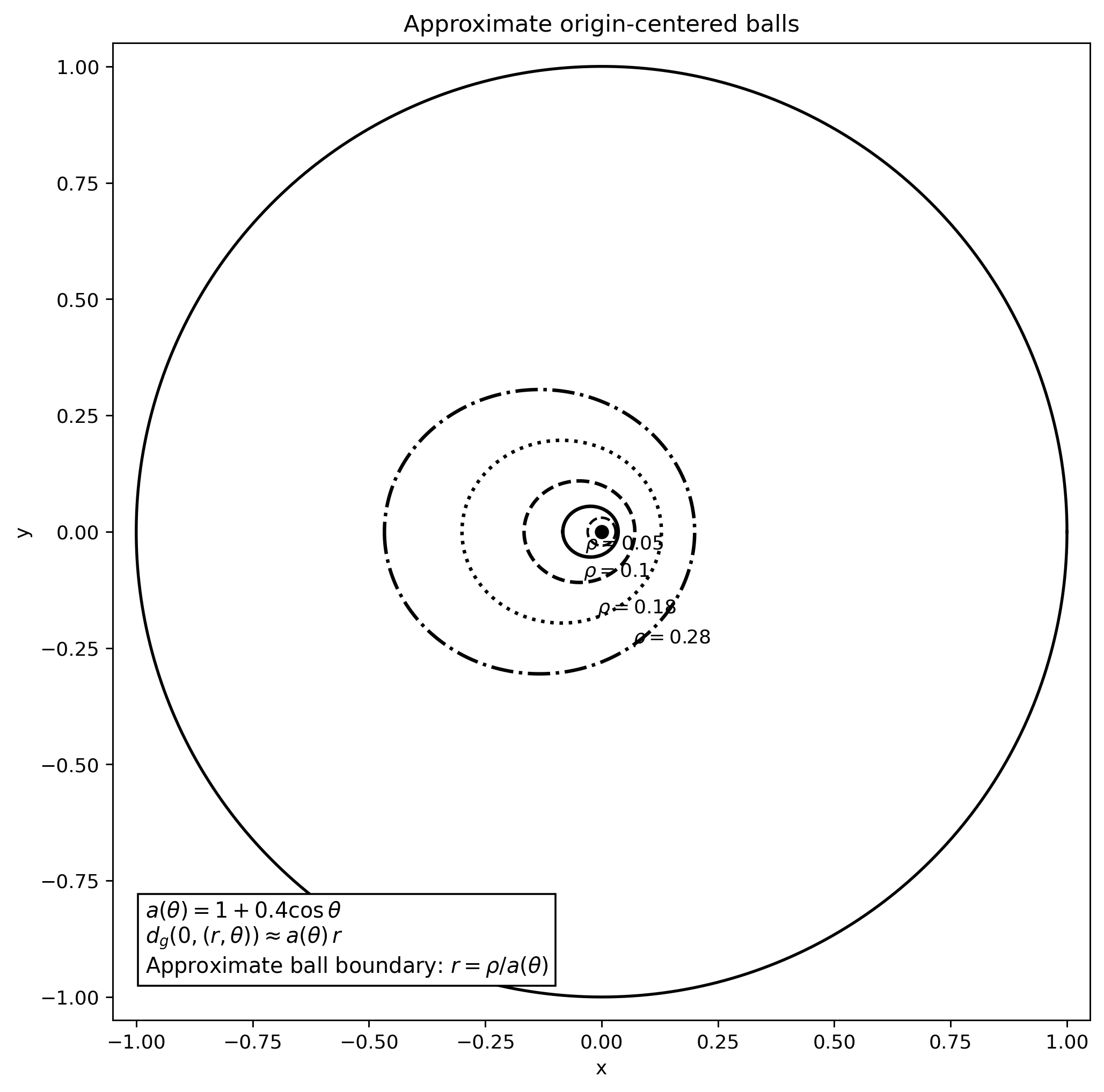}
\caption{Approximate origin-centered metric balls for the angularly conformal metric
\(
g=a(\theta)^2(dr^2+r^2d\theta^2)
\)
with
\(
a(\theta)=1+0.4\cos\theta.
\)
The boundaries shown correspond to the approximation
\(
d_g(0,(r,\theta))\approx a(\theta)\,r
\),
hence are given by
\(
r=\rho/a(\theta)
\)
for several values of $\rho$.}
\label{fig:local-approx-balls-origin}
\end{figure}

To further illustrate the geometry, we also plot several numerical geodesic rays in the punctured disk. Since the metric is singular at the origin, these rays are launched from a small circle of radius $0.05$ around the puncture and are then evolved by numerically solving the geodesic equation for the metric \(g\). Figure \ref{fig:local-geodesic-rays} compares these numerical geodesic rays with the corresponding Euclidean straight radial segments. 

\nd Now for a radial curve $s\mapsto (r(s), \theta_0)$ to be geodesic, we must have $a'(\theta_0)=0.$ In the present example,
\[
a(\theta)=1+0.4\cos\theta,
\qquad
a'(\theta)=-0.4\sin\theta,
\]
so this happens exactly for
\(
\theta_0=0,\pi.
\)
Hence only the rays launched in the directions \(\theta_0=0\) and \(\theta_0=\pi\) remain radial.
 Since in the present example
\[
a(\theta)=1+0.4\cos\theta,
\qquad
a'(\theta)=-0.4\sin\theta,
\]
this happens exactly for \(\theta_0=0\) and \(\theta_0=\pi\). Hence the rays launched in those directions remain radial, whereas rays started at noncritical angles bend, as seen in Figure \ref{fig:local-geodesic-rays}.

\begin{figure}[t]
\centering
\includegraphics[width=0.5\textwidth]{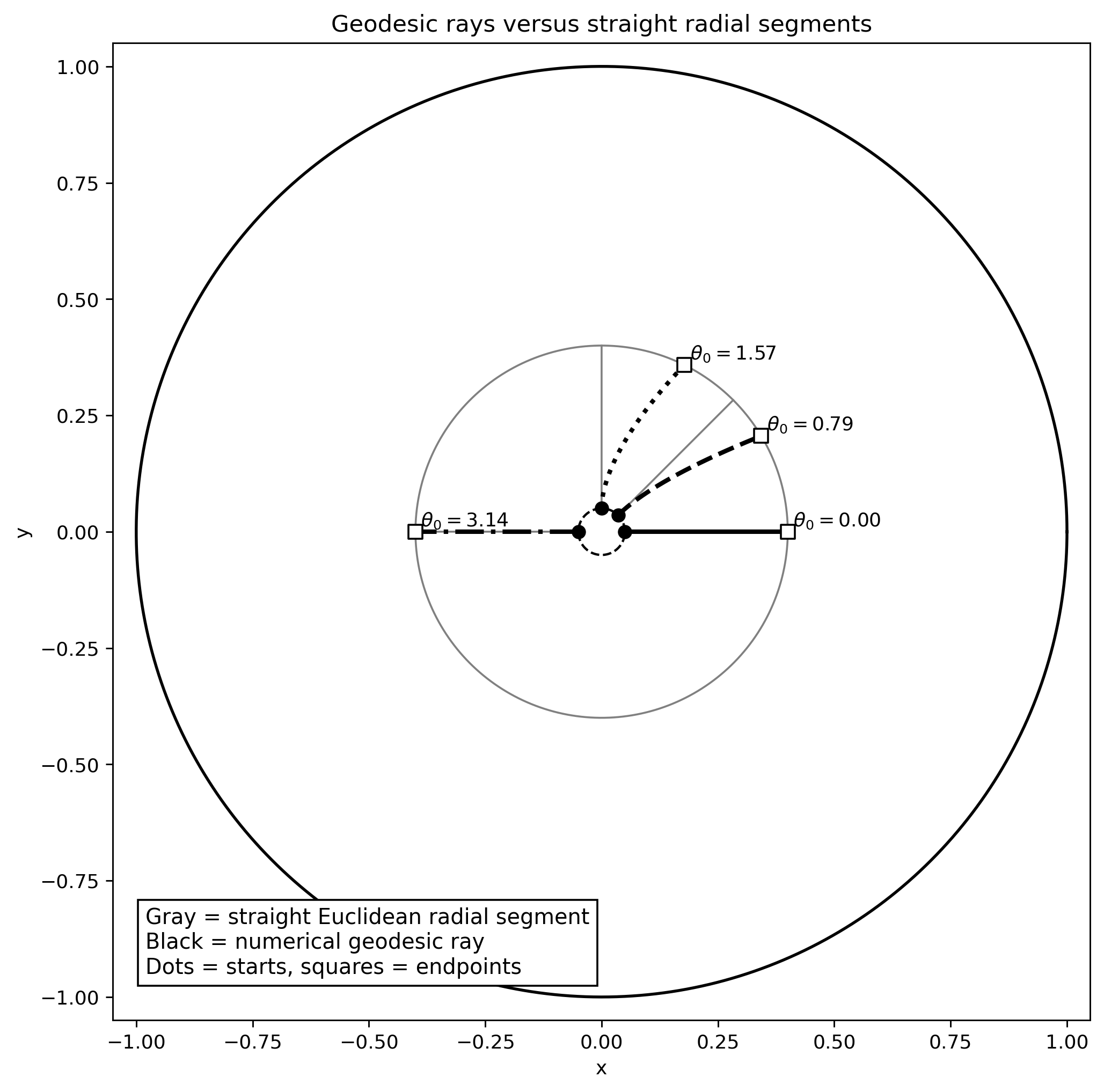}
\caption{Numerical geodesic rays for the metric
\(
g=a(\theta)^2(dr^2+r^2d\theta^2)
\),
compared with straight Euclidean radial segments. The rays are launched from a small circle around the puncture. The comparison shows that the directions \(\theta_0=0\) and \(\theta_0=\pi\) remain radial, whereas rays started at noncritical angles bend under the influence of the angularly varying conformal factor.}
\label{fig:local-geodesic-rays}
\end{figure}


\subsubsection{\textbf{Asymptotic behavior of graph Laplacian}} In the numerical evaluation, we \textit{approximate} the intrinsic distance $d_g(0,(r,\theta))$ by the length of the radial geodesic at fixed $\theta$, namely $d_g(0,(r,\theta))\approx a(\theta)\,r$: equivalently, we restrict the kernel to \textit{radial geodesics} when evaluating the intrinsic graph Laplacian. 
The \textit{test function} is chosen as 
\(
f(x,y) = 1.2x + 0.7y + 0.05(x^2 - 0.5y^2),
\)
so that $\nabla f(0) = (1.2,0.7)$. With the above approximation, the \textit{intrinsic} graph Laplacian is evaluated at the origin using
a quadrature rule for several small values of~$t$:
\[
L_t f(0) \approx \frac{1}{t^{2}}
\int_{0}^{2\pi}\!\!\int_{0}^{1}
e^{-\frac{a(\theta)^2 r^2}{t}}
\bigl(f(0)-f(r\cos\theta,r\sin\theta)\bigr)
\,a(\theta)^2\,r\,dr\,d\theta.
\]
(Above, we truncate the radial integral at $r=1$ since, for small $t$, the exponential factor
$e^{-a(\theta)^2 r^2/t}$ makes the contribution from larger $r$ negligible.)

Table~\ref{tab:scaled_GL} numerically shows that, for $d:=2$ here,
\[
\sqrt{t}\,L_t f(0)
\;\longrightarrow\;
-\,C_2\,\Big\langle \nabla f(0),
\int_{\mathbb{S}^1}\frac{(\cos \theta,\sin \theta)}{a(\theta)}\,d\theta \Big\rangle,
\qquad C_2 = \tfrac{\sqrt{\pi}}{4}.
\]
(Note that although Theorem \ref{thm:blowup of GL near isolated singularity} involves the factor $L(\theta)^{-(d+1)}$, in the present two-dimensional conformal example $L(\theta)=a(\theta)$ and $e^{\frac d2\Psi_1(\theta)}=a(\theta)^2$, so that the combined weight is $a(\theta)^2 a(\theta)^{-3}=a(\theta)^{-1}$, which is exactly the factor appearing in the limit above.)
The integral over $\mathbb{S}^1$ yields
$\displaystyle \int_{\mathbb{S}^1}\frac{(\cos \theta,\sin \theta)}{a(\theta)}\,d\theta = (-1.4308,0)$.
Table~\ref{tab:scaled_GL} compares the computed values of $\sqrt{t}\,L_t f(0)$ with the theoretical limiting constant, together with the associated absolute and relative errors for the $t$-values below.

\begin{table}[h!]
\centering
\caption{Scaled continuous graph Laplacian versus the theoretical limiting constant
($p\equiv1$, $a(\theta)=1+0.4\cos\theta$).}
\label{tab:scaled_GL}
\begin{tabular}{ccccc}
\toprule
$\text{Bandwidth }t$ & $\sqrt{t}\,L_t f(0)$ & Theoretical limit & Abs.\ error & Rel.\ error \\
\midrule
$1.0\times10^{-1}$ & 0.682163 & 0.760824 & 0.078661 & 0.1034 \\
$5.0\times10^{-2}$ & 0.743641 & 0.760824 & 0.017183 & 0.0226 \\
$2.0\times10^{-2}$ & 0.750940 & 0.760824 & 0.009884 & 0.0130 \\
$1.0\times10^{-2}$ & 0.753835 & 0.760824 & 0.006989 & 0.0092 \\
$5.0\times10^{-3}$ & 0.755882 & 0.760824 & 0.004942 & 0.0065 \\
$2.0\times10^{-3}$ & 0.757982 & 0.760824 & 0.002842 & 0.0037 \\
$1.0\times10^{-3}$ & 0.758995 & 0.760824 & 0.001829 & 0.0024 \\
$5.0\times10^{-4}$ & 0.759886 & 0.760824 & 0.000938 & 0.0012 \\
\bottomrule
\end{tabular}
\end{table}

\nd The results confirm that $\sqrt{t}\,L_t f(0)$ converges to the theoretical limiting constant
$-C_2\left\langle\nabla f(0),
\int_{\mathbb{S}^1}(\cos \theta,\sin \theta)/a(\theta)\,d\theta\right\rangle$ as $t\downarrow0$.
This provides a direct numerical verification of the limiting behavior established in Theorem \ref{thm:blowup of GL near isolated singularity} for $d=2.$ Figure \ref{fig:table1-as-graph}
shows the same data graphically.

\begin{figure}[t]
\centering
\includegraphics[width=0.72\textwidth]{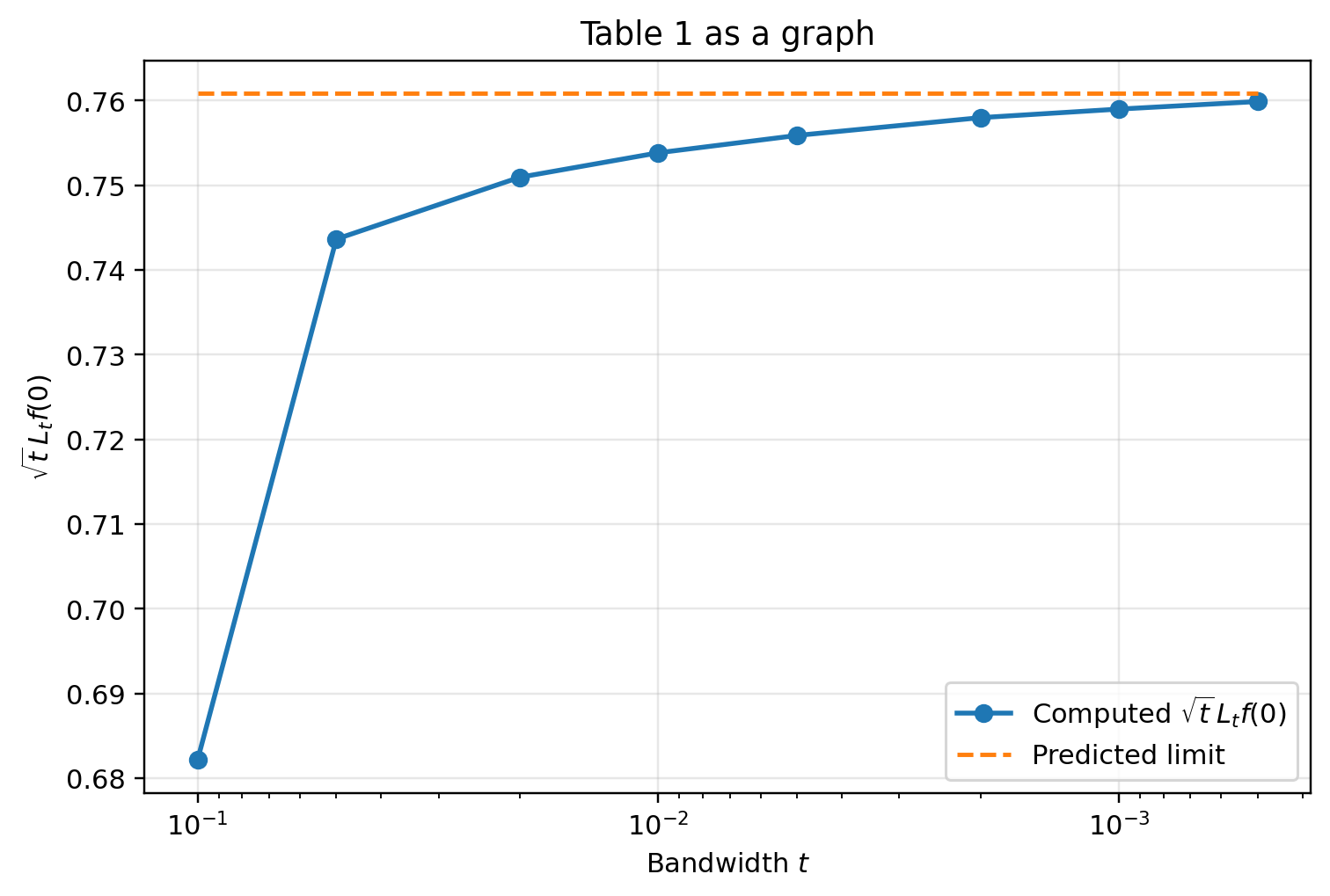}
\caption{Graphical representation of Table~\ref{tab:scaled_GL}. The computed values of \(\sqrt{t}\,L_tf(0)\) are plotted against the bandwidth \(t\), together with the theoretical limiting constant. The legend label ``Predicted limit'' in the plot refers to this theoretical limiting constant. The plot shows convergence of the numerical values to the theoretical constant as \(t \downarrow 0\).}
\label{fig:table1-as-graph}
\end{figure}

\subsection{Simulations on a Cone at its apex}

Using the parametrization $X(u,\theta)=(u\cos\theta,u\sin\theta,u)$ with $u>0$,
$\theta\in[0,2\pi)$, we have $\partial_u X=(\cos\theta,\sin\theta,1)$ and $\partial_\theta X=(-u\sin\theta,u\cos\theta,0)$, hence $g_{uu}=\langle\partial_u X,\partial_u X\rangle=2$, $g_{u\theta}=\langle\partial_u X,\partial_\theta X\rangle=0$, and $g_{\theta\theta}=\langle\partial_\theta X,\partial_\theta X\rangle=u^2$. Therefore the induced metric and area element are
\[
g \;=\; 2\,du^2 + u^2\,d\theta^2, 
\qquad
d\!\operatorname{vol}_g \;=\; \sqrt{2}\,u\,du\,d\theta.
\]

\begin{figure}[t]
    \centering
    \includegraphics[width=0.36\textwidth]{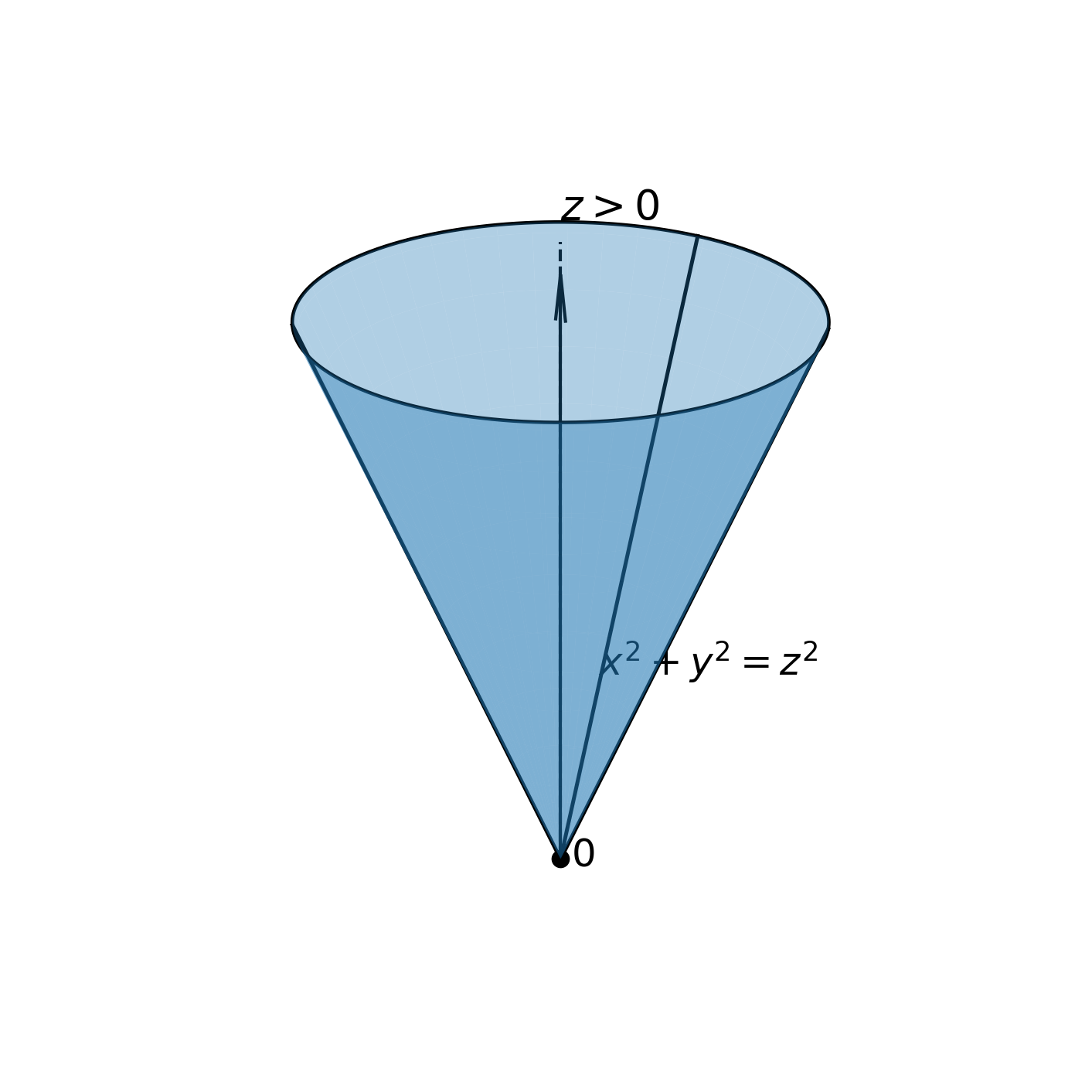}
    \caption{Simulation setup for the cone
    \(
    M=\{(x,y,z)\in\mathbb{R}^3:\ x^2+y^2=z^2,\ z>0\},
    \)
   }
    \label{fig:cone_apex_experiment}
\end{figure}

\nd The metric $g$ is defined on $M\setminus\{0\}$ and \emph{does not} extend smoothly at the apex.
In particular, the integration for the Gaussian operator below is over the \textit{regular part} $M\setminus\{0\}$; the point
$x=0$ (the apex) is used only as the \emph{evaluation point} of the operator.

\nd We consider the \emph{extrinsic kernel} Gaussian operator at $x=0$ with $d:=2$ here:
\[
L_t f(0)
\;=\; \frac{1}{t^{2}}\int_{M\setminus \{0\}}
e^{-\frac{\|y\|^2}{t}}\,(f(0)-f(y))\,d\!\operatorname{vol}_g(y),
\]
where $\|y\|$ is the Euclidean norm in $\mathbb{R}^3$.
Choosing $f(x,y,z)=x^2+y^2$ (so $f(0)=0$ and $f\circ X(u,\theta)=u^2$), we obtain the explicit formula for the \textit{extrinsic} graph Laplacian:
\[
L_t f(0)
\;=\;\frac{1}{t^{2}}\int_{0}^{2\pi}\!\!\int_{0}^{\infty}
e^{-\frac{2u^2}{t}}\,(0-u^2)\;\sqrt{2}\,u\;du\,d\theta.
\]
A direct calculation shows that this value is \emph{independent of $t$}:
\[
L_t f(0)\;=\;-\frac{\pi\sqrt{2}}{4}\;\approx\;-1.1107207345.
\]

\nd Thus, the asymptotics are exact in this example: \(L_t f(0)\) is independent of \(t\), so in particular \(\sqrt{t}\,L_t f(0)\to 0\) as \(t\downarrow 0\). Hence $L_tf$ does not become unbounded at the singularity as $t\to0$, contrary to the naive expectation that every non-removable singularity should produce blow-up of the continuous graph Laplacian, as happens for boundary-type singularities treated in \citet{BelkinQueWangZhou2012} and \citet{PalTewodrose2025}. In particular here, $\sqrt{t}\,L_t f(0)\to 0$ as $t\downarrow 0$, contrasting with the previous simulation in Table \ref{tab:scaled_GL}, where the same limit is non-zero.

\noindent\textbf{Clarification (domain vs.\ evaluation point).}
Throughout this experiment, the metric $g=2\,du^2+u^2d\theta^2$ and area element $d\!\operatorname{vol}_g=\sqrt{2}\,u\,du\,d\theta$
are defined on $M\setminus\{0\}$. The integral defining $L_t f(0)$ is taken over
$M\setminus\{0\}$, while $x=0$ is the \emph{evaluation point} in the kernel $e^{-\|x-y\|^2/t}$. This distinction is essential and
matches the non-removable nature of the cone tip (angle deficit), even though the Gaussian curvature is $0$ away from the tip. See Definition/Remark \ref{def:GL for loc. ang. conformal metrics} for a comparison, where the locally angularly conformal metrics (cf. Definition \ref{def:locally angularly conformal}) were also \textit{not} defined at the isolated singularity. Figure~\ref{fig:cone_apex_experiment} illustrates the simulation setup for the cone used in this example.

\bigskip

\section{Declarations}
\subsection{Funding}
This research was funded by the Research Foundation–Flanders (FWO)
via the Odysseus II programme no. G0DBZ23N.
\subsection{Competing interests/disclosure}
The authors declare no competing interests.
\subsection{Data Availability Statement}
The data that support the findings of this study are available from the corresponding author, [S.P.], upon reasonable request.
\subsection{Generative Artificial Intelligence (AI)}
The author used ChatGPT (OpenAI) for limited assistance with language and LaTeX polishing and for preliminary literature-search support. All mathematical content, proofs, results, scientific conclusions, and bibliographic references were independently checked and verified by the author.


\bibliographystyle{plainnat}   
\bibliography{references}

\end{document}